\documentclass[12pt]{article}

\usepackage{mathtools} 
    \mathtoolsset{%
        mathic = true   
    }%
    \numberwithin{equation}{section}   
\usepackage{amsthm}   
\usepackage{xcolor}   
\usepackage{enumitem}   
    \setenumerate{listparindent=\parindent}   

\usepackage[T1]{fontenc}
\usepackage[main=english,french]{babel}

\usepackage{listings} 
    \lstdefinestyle{sagecode}{
        language=Python,
        basicstyle=\ttfamily\footnotesize,
        numbers=left,
        numberstyle=\tiny,
        numbersep=8pt,
        breaklines=true,
        columns=fullflexible,
        keepspaces=true,
        showstringspaces=false,
        frame=single
    }

\usepackage{array}   
    \newcolumntype{L}{>{$}l<{$}}   
    \newcolumntype{C}{>{$}c<{$}}   
    \newcolumntype{R}{>{$}r<{$}}   

\usepackage{amssymb}   
\usepackage{mathrsfs}   

\usepackage{centernot} 

\usepackage{tikz}

\usepackage{csquotes}

\usepackage[
    backend=biber,
    style=alphabetic, maxalphanames=4, minalphanames=3, 
    autolang=hyphen,
    maxbibnames=99,
    isbn=false,
]{biblatex}
    \AtEveryBibitem{%
        \clearfield{note} 
        \iffieldundef{doi} 
            {}
            {\clearfield{url}\clearfield{urldate}}%
        \iffieldundef{eprint}
            {}
            {\clearfield{url}\clearfield{urldate}}%
    }
    \renewbibmacro*{in:}{} 
    \renewbibmacro*{issue+date}{
        \ifboolexpr{
            test {\iffieldundef{issue}}
            and
            test {\iffieldundef{year}}
        }
            {}
            {\printtext[parens]{%
                \printfield{issue}%
                \setunit*{\addspace}%
                \usebibmacro{date}}}%
        \newunit}
    \renewbibmacro*{journal}{
        \iffieldundef{shortjournal}
            {\printtext[journaltitle]{%
                \printfield[titlecase]{journaltitle}%
                \setunit{\subtitlepunct}%
                \printfield[titlecase]{journalsubtitle}}}
            {\printtext[journaltitle]{\printfield{shortjournal}}}%
        \newunit
        \printfield{journaltitleaddon}}
    
\usepackage{hyperref} 
    \hypersetup{%
        bookmarksopen  = true,    
        bookmarksdepth = 1,       
        colorlinks     = true,    
        urlcolor      = blue, 
        linkcolor     = blue, 
        citecolor     = blue, 
    }%

\usepackage{mathscinet} 

\usepackage[capitalize,noabbrev]{cleveref} 
    \crefname{proposition}{Proposition}{Propositions}   
    \crefname{const}{Construction}{Constructions}
    \crefname{subthm}{Theorem}{Theorems}   
    \crefname{sublem}{Lemma}{Lemmas}   
    \crefname{subprop}{Proposition}{Propositions}   
    \crefname{subcor}{Corollary}{Corollaries}   
    \crefname{subdef}{Definition}{Definitions}   
    \crefname{cond}{Condition}{Conditions}   
    \crefname{conj}{Conjecture}{Conjectures}
    \crefname{part}{Part}{Parts}   
    \crefname{problem}{Problem}{Problems}
    \crefname{exercise}{Exercise}{Exercises}
    \crefname{case}{Case}{Cases} 
    \crefformat{ineq}{Inequality~(#2#1#3)}   
        \crefrangeformat{ineq}{Inequalities~(#3#1#4) to~(#5#2#6)}
        \crefmultiformat{ineq}{Inequalities~(#2#1#3)}%
            { and~(#2#1#3)}{, (#2#1#3)}{ and~(#2#1#3)}
    \crefformat{isom}{Isomorphism~(#2#1#3)}   
        \crefrangeformat{isom}{Isomorphisms~(#3#1#4) to~(#5#2#6)}
        \crefmultiformat{isom}{Isomorphisms~(#2#1#3)}%
            { and~(#2#1#3)}{, (#2#1#3)}{ and~(#2#1#3)}
    \crefformat{eqns}{Equations~(#2#1#3)}   
    \crefformat{containment}{Containment~(#2#1#3)}   
        \crefrangeformat{containment}{Containments~(#3#1#4) to~(#5#2#6)}
        \crefmultiformat{containment}{Containments~(#2#1#3)}%
            { and~(#2#1#3)}{, (#2#1#3)}{ and~(#2#1#3)}
    \crefformat{containments}{Containments~(#2#1#3)}   
    \AddToHook{env/proposition/begin}{\crefalias{theorem}{proposition}}
    \AddToHook{env/lemma/begin}{\crefalias{theorem}{lemma}}
    \AddToHook{env/remark/begin}{\crefalias{theorem}{remark}}
    \AddToHook{env/corollary/begin}{\crefalias{theorem}{corollary}}
    \AddToHook{env/definition/begin}{\crefalias{theorem}{definition}}
    \AddToHook{env/example/begin}{\crefalias{theorem}{example}}
    \makeatletter
    \AddToHook{begindocument/end}{%
        \def\ltx@label#1{%
            \my@ams@cref@label#1\@nil
        }%
    
        \def\my@ams@cref@label{%
            \@ifnextchar[%
                {\my@ams@cref@label@opt}%
                {\my@ams@cref@label@plain}%
        }%
    
        \def\my@ams@cref@label@opt[#1]#2\@nil{%
            \cref@label[#1]{#2}%
        }%
    
        \def\my@ams@cref@label@plain#1\@nil{%
            \cref@label{#1}%
        }%
    }
    \makeatother

\usepackage{crossreftools}
\DeclarePairedDelimiter{\nparens}{\lparen}{\rparen} 
\newcommand{\bigparens}{\nparens[\big]} 
\DeclarePairedDelimiter{\nverts}{\lvert}{\rvert} 

\newcommand{\ncard}{\nverts} 
\newcommand{\card}{\ncard*} 

\newcommand{\stsymbol}[1]{#1\vert}      
\newcommand{\st}{} 
\newcommand{\nst}[1][]{
    \mathchoice{\;}{\;}{\,}{\,} 
    \stsymbol{#1}
    \allowbreak 
    \mathchoice{\;}{\;}{\,}{\,} 
    \mathopen{} 
}%

\DeclarePairedDelimiterX{\nsetof}[1]{\lbrace}{\rbrace}{
    \renewcommand{\st}{\nst[\delimsize]} 
    #1
}
\newcommand{\setof}{\nsetof*} 

\DeclarePairedDelimiterX{\ngenby}[1]{\langle}{\rangle}{
    \renewcommand{\st}{\nst[\delimsize]} 
    #1
}

\newcommand{\defing}[1]{\textup{\textbf{#1}}} 
\providecommand{\keywords}[1]{
    \par\medskip
    \noindent\textbf{Keywords:} #1\par
}
\providecommand{\subjclass}[1]{
    \par\smallskip
    \noindent\textbf{2020 Mathematics Subject Classification:} #1\par
}

\newcommand{\Z}{\mathbb{Z}} 
\newcommand{\Znn}{\Z_{\ge 0}} 
\newcommand{\Zp}{\Z_{\ge 1}} 
\newcommand{\Q}{\mathbb{Q}} 
\newcommand{\R}{\mathbb{R}} 
\providecommand{\C}{\mathbb{C}} 

\DeclareMathOperator{\lcm}{lcm} 

\newcommand{\deftobe}{\coloneq} 

\newcommand{\tensor}{\otimes}
\newcommand{\maps}{\colon} 

\newcommand{\resto}[2]{{
   \left.\kern-\nulldelimiterspace 
   #1 
   \right\vert_{#2} 
}}

\renewcommand{\epsilon}{\varepsilon}
\renewcommand{\subset}{\subseteq}

\renewcommand*{\mod}[1]{\;\operatorname{mod} {#1}}

\renewcommand{\phi}{\varphi} 

\let\isom\cong 
\renewcommand{\cong}{\equiv} 

\makeatletter
\newcommand\RedeclareMathOperator{%
    \@ifstar{\def\rmo@s{m}\rmo@redeclare}{\def\rmo@s{o}\rmo@redeclare}%
}
\newcommand\rmo@redeclare[2]{%
    \begingroup \escapechar\m@ne\xdef\@gtempa{{\string#1}}\endgroup
    \expandafter\@ifundefined\@gtempa
        {\@latex@error{\noexpand#1undefined}\@ehc}%
        \relax
    \expandafter\rmo@declmathop\rmo@s{#1}{#2}}
\newcommand\rmo@declmathop[3]{%
    \DeclareRobustCommand{#2}{\qopname\newmcodes@#1{#3}}%
}
\@onlypreamble\RedeclareMathOperator
\makeatother

\makeatletter
\DeclareRobustCommand{\rvdots}{%
    \vbox{%
        \baselineskip4\p@\lineskiplimit\z@
        \kern-\p@
        \hbox{.}\hbox{.}\hbox{.}
    }%
}%
\makeatother

\DeclareFontFamily{U}{mathb}{\hyphenchar\font45}
\DeclareFontShape{U}{mathb}{m}{n}{
      <5> <6> <7> <8> <9> <10> gen * mathb
      <10.95> mathb10 <12> <14.4> <17.28> <20.74> <24.88> mathb12
      }{}
\DeclareSymbolFont{mathb}{U}{mathb}{m}{n}
\DeclareMathSymbol{\lefttorightarrow}      {3}{mathb}{"FC}
\DeclareMathOperator{\spn}{span}

\newcommand*{\dualvs}[1]{#1^{\ast\!}} 
\newcommand*{\trans}[1]{#1^{T}} 

\newcommand{\g}{\ensuremath{\mathfrak{g}}}   
\newcommand{\h}{\ensuremath{\mathfrak{h}}}   
\newcommand{\ad}{\operatorname{ad}}   
\newcommand{\hd}{\dualvs{\h}}   
\newcommand{\roots}{\Phi}   
\newcommand*{\co}[1]{#1^{\vee}}   
\newcommand{\posroots}{\roots^{+}}   
\newcommand{\weights}{\Lambda}   
\newcommand{\Wg}{\mathfrak{W}}   
\newcommand{\An}{\mathrm{A}_{n}}   
\newcommand{\Bn}{\mathrm{B}_{n}}   
\newcommand{\Cn}{\mathrm{C}_{n}}   
\newcommand{\Dn}{\mathrm{D}_{n}}   
\newcommand{\En}{\mathrm{E}}   
\newcommand{\F}{\mathrm{F}_{4}}   
\newcommand{\G}{\mathrm{G}_{2}}   
\newcommand{\Wc}{\mathcal{W}} 

\DeclareMathOperator{\conv}{conv} 
\DeclareMathOperator{\vertx}{vert} 

\newcommand*{\ehr}[1]{\operatorname{ehr}_{#1}} 
\DeclareMathOperator{\den}{den} 

\theoremstyle{definition}   
\newtheorem{theorem}{Theorem}[section]   

\newtheorem*{theorem*}{Theorem}
\newtheorem{lemma}[theorem]{Lemma}

\newtheorem*{lemma*}{Lemma}

\newtheorem{proposition}[theorem]{Proposition}
\newtheorem*{proposition*}{Proposition}

\newtheorem{conjecture}{Conjecture}

\newcommand{\thistheoremname}{}
\newtheorem*{genericthm}{\thistheoremname}

\theoremstyle{definition}
\newtheorem{example}[theorem]{Example}

\newtheorem{definition}[theorem]{Definition}

\newtheorem*{definition*}{Definition}

\newtheorem*{correct}{CORRECTION}

\theoremstyle{remark}

\newtheorem*{claim*}{Claim}

\newenvironment{defnconditions}{%
    \begin{enumerate}[label=(\roman*), ref={(\roman*)}]
}{%
    \end{enumerate}
}%

\pgfdeclarelayer{background}
\pgfdeclarelayer{foreground}
\pgfdeclarelayer{lattice}
\pgfdeclarelayer{axes}
\pgfdeclarelayer{polygon1}
\pgfsetlayers{background, main, polygon1, axes, lattice, foreground}

\tikzset{%
    every picture/.style={scale=1},
    blackpoint/.style={circle, inner xsep=1.5pt, inner ysep=0pt, fill=black},
    graypoint/.style={circle, inner xsep=3pt, inner ysep=0pt, fill=gray},
    redpoint/.style={circle, inner xsep=3pt, inner ysep=0pt, fill=red},
    bluepoint/.style={circle, inner xsep=3pt, inner ysep=0pt, fill=blue},
    axisstyle/.style={-latex, line width = 1pt},
    arrowstyle/.style={-latex, line width = 1pt},
    polygonstyle/.style={fill=black!25, thick, rounded corners=0.001pt},
}

\definecolor{ione}{HTML}{0F6E56}
\definecolor{itwo}{HTML}{C0501F}
\usetikzlibrary{arrows.meta}
\tikzset{wt/.style={black!45},trailstep/.style={-{Stealth[length=5pt]},line width=1.1pt,shorten <=2.5pt,shorten >=2.5pt},lab/.style={font=\scriptsize,fill=white,inner sep=0.8pt}}
\tikzset{seg/.style={line width=1.2pt,line cap=round}}
\definecolor{pathred}{HTML}{C62828}

\newcommand{\unibi}{\unichar{"1D422}}   
\newcommand{\unisubtwo}{\unichar{"2082}}   

\usepackage{booktabs}
\usepackage{caption, subcaption}
\usepackage{float}
\usepackage{physics}
\usepackage{longtable}

\newcommand{\liebracket}[2]{[{#1}\ {#2}]}
\newcommand{\wa}{w_{1}}   
\newcommand{\wb}{w_{2}}   
\newcommand{\rf}{\spn_{\R}(\roots)}   
\newcommand{\ld}[1]{^{L}#1} 
\newcommand{\clnm}{\ensuremath{c_{\lambda,\nu}^{\mu}}}
\newcommand{\lng}[1]{\ell(#1)} 
\newcommand{\bz}{\mathrm{BZ}_{\lambda\nu}^{\mu}(\g)}  
\newcommand{\bzlmn}{\bz}  
\newcommand{\bzlmnc}{\text{c-}\bzlmn}  
\newcommand{\bzlmnd}{\text{d-}\bzlmn}  
\newcommand{\pip}{T_\lambda^{\gamma,\delta}}   
\renewcommand{\i}{\ensuremath{\mathbf{i}}} 
\newcommand*{\DD}{\mathcal{D}} 
\newcommand{\PP}{\mathcal{P}} 
\renewcommand{\P}{\PP} 
\newcommand{\FF}{\mathcal{F}} 
\DeclareMathOperator{\Div}{Div} 
\newcommand{\II}{\mathcal{I}} 
\newcommand{\posrootsmatrix}{[\roots^{+}]}   

\newcommand{\sagemath}{SageMath}
\newcommand{\TikZ}{\textup{Ti\textit{k}Z}}

\AtBeginBibliography{
    \setlength{\emergencystretch}{1em}%
}

\title{%
    Berenstein--Zelevinsky polytopes for type \(\G\)
}%

\author{%
    Thomas Magnuson%
    \thanks{
        \href{mailto:tmagnus1@uwyo.edu}{
            \nolinkurl{tmagnus1@uwyo.edu}
        }
    }
    \and 
    Tyrrell McAllister%
    \thanks{
        \href{mailto:tmcallis@uwyo.edu}{
            \nolinkurl{tmcallis@uwyo.edu}
        }
    }
}%

\date{}

\begin{document}

\maketitle

\begin{abstract}
    In 2001, Berenstein and Zelevinsky presented a polyhedral model for the problem of computing tensor-product multiplicities for semi\-simple Lie algebras. We provide explicit linear equations and inequalities for Berenstein--Zelevinsky (BZ) polytopes of type-$\mathrm{G}_{2}$ Lie algebras. We prove that the denominator of a type-$\mathrm{G}_{2}$ BZ~polytope divides $180$, and that this bound is sharp, while the quasi-period of the Ehrhart quasipolynomial of such a polytope is always a divisor of $6$. We also show that BZ~polytopes of type-$\mathrm{G}_{2}$ Lie algebras can fail to satisfy Ehrhart positivity, in contrast to a conjectured property of BZ~polytopes for all classical types.
\end{abstract}

\keywords{%
    Berenstein--Zelevinsky polytopes,
    tensor product multiplicities,
    Ehrhart theory
}

\subjclass{%
    17B10; 52B05
}

\section{Introduction}
\label{sec: introduction}

Let $\g$ be a finite-dimensional complex semisimple Lie algebra. The \emph{Clebsch--Gordan problem} is to compute the multiplicity $\clnm$ with which the finite-dimensional irreducible $\g$-module $V_{\mu}$ appears in the tensor product $V_{\lambda} \tensor V_{\nu}$, where $\lambda$, $\mu$, and $\nu$ are dominant weights of $\g$. Berenstein and Zelevinsky developed a \emph{polyhedral model} for this problem in a series of papers culminating in \cite{BerZel2001}. They constructed families $\setof{\PP_{\lambda,\nu}^{\mu}}$ of rational polytopes, parametrized by triples $(\lambda, \mu, \nu)$ of dominant weights of $\g$, in which $\clnm$ is precisely the number of integer-lattice points: $\card{\PP_{\lambda,\nu}^{\mu} \cap \Z^{\ell}} = \clnm$. We call the polytopes in these families \emph{BZ polytopes}.  (See \cref{sec: background} for precise definitions.)

For the cases in which $\g$ is of classical type $\An, \Bn, \Cn,$ or $ \Dn$, Berenstein and Zelevinsky provided completely explicit systems of linear equations and inequalities for the BZ polytopes, with the type-$\An$ case appearing in \cite[Section~2]{BerZel1992} and the remaining classical types in \cite[Section~2.3]{BerZel2001}. However, for the exceptional types $\En_{6}$, $\En_{7}$, $\En_{8}$, $\F,$ and $ \G$, the definition of the BZ~polytopes in \cite{BerZel2001} is not entirely explicit, as this definition is stated in terms of so-called \emph{$\i$\nobreakdash-trails}, which in turn are not given a purely combinatorial definition.

Fix a reduced word $\i$ for the longest element of the Weyl group of $\g$. Roughly speaking, a sequence $\pi$ of weights of a $\g$-module is an $\i$-trail if $\pi$ satisfies two conditions. The first condition (\cref{defn: potential i-trail} below) is purely combinatorial. When this first condition alone is imposed, we call $\pi$ a \emph{potential $\i$-trail}. To be an (actual) $\i$-trail, the sequence $\pi$ must satisfy a second, non-combinatorial, condition that we call the \emph{composition rule} (\cref{defn: i-trail} below).

In \cref{subsec: Triviality of the composition rule in type G2} below, we prove that, in type $\G$, \emph{every} potential $\i$-trail arising in the definition of the BZ polytopes satisfies the composition rule and so is an $\i$-trail.

\begin{theorem}\label{thm: in G2 every relevant potential i-trail is an actual i-trail}
    Let $\g$ be of type $\G$ and let $\i$ be a reduced word for the longest element of the Weyl group of $\g$. Then every potential $\i$-trail arising in the definitions of the BZ polytopes in \cref{defn: c-vec polytope,defn: d-vec polytope} below is an $\i$-trail.
\end{theorem}

As a result of this theorem, the $\i$-trails required for type-$\G$ BZ polytopes have an entirely combinatorial description. In \cref{subsec: potential i-trails as lattice points}, we use this description to express the potential $\i$-trails in a $\g$-module~$V_{\lambda}$ from a weight~$\gamma$ to a weight~$\delta$ as the lattice points in a particular rational polytope $\pip$. In type~$\G$, lattice-point enumeration algorithms then provide an effective means of computing all $\i$-trails appearing in the definitions of BZ polytopes. We list these $\i$-trails in \cref{subsec: Potential i-trails for type-G2 BZ polytopes}.%
\footnote{
    Late in the preparation of this manuscript, we learned that Kanakubo, Koshevoy, and Nakashima \cite{KKN2022} show that the monomials of certain generalized minors are in bijection with $\i$-trails. Further, for type $\G$ these authors compute the monomials and explain how to recover the relevant $\i$-trails. However, to the best of our knowledge, they do not establish that every relevant potential $\i$-trail in type $\G$ is an actual $\i$-trail.
}

The list of these $\i$-trails in \cref{subsec: Potential i-trails for type-G2 BZ polytopes} allows us to write the BZ~polytopes for type~$\G$ as an explicit system of linear equations and inequalities. In particular, the BZ polytopes defined in \cref{defn: c-vec polytope} with $\i = (1,2,1,2,1,2)$ are given in the following theorem.

\begin{theorem}\label{thm: explicit equations and inequalities for type G2}
    Let $\g$ be of type \(\G\) and let $(\lambda_{1},\lambda_{2})$, $(\mu_{1}, \mu_{2})$, and $(\nu_{1}, \nu_{2})$ be dominant weights of $\g$ written in the basis of fundamental weights. Then $\clnm$ is the number of integral $6$-tuples $(t_{1}, \dotsc, t_{6}) \in \Z^{6}$ such that
    \begin{enumerate}
        \item
        $t_{1} \ge 0$ and \\ $6t_{2} \ge 2t_{3} \ge 3t_{4} \ge 2t_{5} \ge 6t_{6} \ge 0$,
        
        \item 
        $t_{1} + t_{3} + t_{5} = 2(\lambda_{1} + \nu_{1} - \mu_{1}) + 3(\lambda_{2} + \nu_{2} - \mu_{2})$, \\
        $t_{2} + t_{4} + t_{6} = \phantom{2}(\lambda_{1} + \nu_{1} - \mu_{1}) + 2(\lambda_{2} + \nu_{2} - \mu_{2})$,
        
        \item 
        $2\lambda_{1} + 6\lambda_{2} \ge 2t_{1} + 6\lambda_{2} \ge 2t_{2} + 4 \lambda_{2} \ge t_{3} + 3\lambda_{2} \ge 2t_{4} + 2\lambda_{2} \ge 2t_{5}$ \\ and $\lambda_{2} \ge t_{6}$,
        
        \item 
        $\nu_{1} 
            \ge
            \max \setof{
                \begin{matrix*}[r]
                    t_{1} - 3t_{2} + 2t_{3} - 3t_{4} + 2t_{5} - 3t_{6}, \\
                    t_{3} - 3t_{4} + 2t_{5} - 3t_{6}, \\
                    t_{5} - 3t_{6}\phantom{,}
                \end{matrix*}
            }$,
        
        \noindent
        $\nu_{2} 
            \ge
            \max \setof{
                \begin{matrix*}[r]
                    t_{2} - t_{3} + 2t_{4} - t_{5} + 2t_{6}, \\
                    t_{4} - t_{5} + 2t_{6}, \\
                    t_{6}\phantom{,}
                \end{matrix*}
            }$.
    \end{enumerate}
\end{theorem}

In \cite{DeLMcA2006}, the second author and De~Loera applied methods from polyhedral geometry, especially Ehrhart theory, to the BZ polytopes of classical type. This approach led to a conjectured generalization of the \emph{Saturation Theorem} of Knutson and Tao.

\begin{theorem}[Saturation Theorem \cite{KnuTao1999}]
    Let $\g$ be of type~$\An$, and let $\lambda, \mu, \nu$ be dominant weights of $\g$. If $c_{t\lambda, t\nu}^{t\mu} > 0$ for some $t \in \Zp$, then $\clnm > 0$.
\end{theorem}

In view of the polyhedral model of $\clnm$ introduced by Berenstein and Zelevinsky, it is very natural to interpret the so-called \emph{stretched Clebsch--Gordan coefficient} $t \mapsto c_{t\lambda, t\nu}^{t\mu}$ as an \emph{Ehrhart quasipolynomial}.

\begin{definition}\label{defn: quasipolynomial}
    A function $\phi \maps \Zp \to \C$ is a \defing{quasipolynomial} if there exist a positive integer $p$ and polynomials $f_{1}, \dotsc, f_{p} \in \C[x]$ such that $\phi(t) = f_{i}(t)$ whenever $t \cong i \mod p$, in which case $p$ is a \defing{quasi-period} of $\phi$. The quasipolynomial is of \defing{degree-$d$} if $d$ is the maximum degree among $f_{1}, \dotsc, f_{p}$.
\end{definition}

\begin{theorem}[\cite{Ehr1962a}]
\label{thm: ehrharts theorem}
    Let $\PP \subset \R^{n}$ be a $d$-dimensional rational polytope. Then the \defing{Ehrhart function} $\ehr{\PP}(t) \deftobe \card{t\PP \cap \Z^{n}}$ ($t \in \Zp$) is a degree-$d$ rational \emph{quasipolynomial} with a minimum quasi-period that divides the denominator of~$\PP$.
\end{theorem}

Derksen and Weyman \cite{DerWey2002} proved that, in type $\An$, the stretched Clebsch--Gordan coefficients $t \mapsto c_{t\lambda, t\nu}^{t\mu}$ are polynomial functions of $t \in \Zp$. (That is, the minimum quasi-period of $\ehr{\P}(t)$ is $p = 1$ when $\P$ is a BZ~polytope of type~$\An$.%
\footnote{%
    In other words, the BZ polytopes in type $\An$ are \emph{pseudo-integral} (cf.~\cite{McAWil2026+}).%
})
King, Tollu, and Toumazet \cite{KTT2004} conjectured that all coefficients of these polynomials are nonnegative. This conjecture, if proved, would trivially imply the Saturation Theorem in type~$\An$. In \cite{DeLMcA2006}, the authors proved that, in the remaining classical types $\Bn, \Cn,$ and $ \Dn$, the stretched Clebsch--Gordan coefficients $t \mapsto c_{t \lambda, t\nu}^{t\mu}$ are quasipolynomials of minimum quasi-period at most $2$. On the basis of many computed examples, these authors extended the conjecture of King, Tollu, and Toumazet to all classical types:

\begin{conjecture}[\cite{DeLMcA2006}]\label[conj]{conj: ehrhart positivity}
    Let $\g$ be of type $\An$, $\Bn$, $\Cn$, or $\Dn$, and let $\lambda, \mu, \nu$ be dominant weights of $\g$. Then all coefficients of the quasipolynomial $t \mapsto c_{t \lambda, t\nu}^{t\mu}$ are nonnegative.
\end{conjecture}

\begin{sloppypar}
    With \cref{thm: explicit equations and inequalities for type G2} in hand, it is straightforward to compute many stretched Clebsch--Gordan coefficients in type~$\G$ using, for example, Barvinok's algorithm as implemented in LattE \cite{LattE2013}, which is accessible through passagemath \cite{passagemath}. In the course of such experiments, one quickly finds that \cref{conj: ehrhart positivity} \emph{can\-not} be extended to type $\G$:
\end{sloppypar}

\begin{example}\label{exm: counterexample to ehrhart positivity}
    Let $\g$ be of type~$\G$, and let $\lambda \deftobe (1, 0)$, $\mu \deftobe (0,1)$, and $\nu \deftobe (0,1)$ (written in the basis of fundamental weights). Let $\P$ be the BZ~polytope defined by the linear system in \cref{thm: explicit equations and inequalities for type G2}. Then
    \begin{equation*}
        c_{t\lambda, t\nu}^{t\mu}
            =
            \ehr{\P}(t)
            =
            \begin{cases*}
                \frac{1}{6}t + 1           & if $t \cong 0 \mod{6}$, \\[1ex]
                \frac{1}{6}t - \frac{1}{6} & if $t \cong 1 \mod{6}$, \\[1ex]
                \frac{1}{6}t + \frac{2}{3} & if $t \cong 2 \mod{6}$, \\[1ex]
                \frac{1}{6}t + \frac{1}{2} & if $t \cong 3 \mod{6}$, \\[1ex]
                \frac{1}{6}t + \frac{1}{3} & if $t \cong 4 \mod{6}$, \\[1ex]
                \frac{1}{6}t + \frac{1}{6} & if $t \cong 5 \mod{6}$.
            \end{cases*}
    \end{equation*}
    The negative coefficient appearing in the $t \cong 1 \mod 6$ case establishes that the analogue of \cref{conj: ehrhart positivity} for type $\G$ does not hold.
\end{example}

In \cref{sec: Quasi-periods of BZ polytopes of type G2} below, we prove that the minimum quasi-period $6$ appearing in \cref{exm: counterexample to ehrhart positivity} is the largest possible for BZ polytopes of type $\G$:
\begin{theorem}\label{thm: quasi-period at most 6}
    Let $\g$ be a type-$\G$ Lie algebra, and let $\lambda$, $\mu$, and $\nu$ be dominant weights of~$\g$. Then the minimum quasi-period of the stretched Clebsch--Gordan coefficient $t \mapsto c_{t\lambda, t\nu}^{t\mu}$ divides $6$.
\end{theorem}

On the other hand, in \cref{sec: denominators of vertices}, we show that the denominators of the type-$\G$ BZ polytopes in \cref{thm: explicit equations and inequalities for type G2} are bounded above by $180$. In particular, we prove the following:
\begin{theorem}\label{thm: possible denominators - intro version}
    Let $\g$ be a type-$\G$ Lie algebra, and let $\lambda$, $\mu$, and $\nu$ be dominant weights of~$\g$. Then $\den(v) \in \setof{1,2,3,4,5,6,9}$ for every vertex $v$ of the BZ polytope defined in \cref{thm: explicit equations and inequalities for type G2}, and hence the denominator of this polytope divides $180$.
\end{theorem}
We prove this result in \cref{sec: denominators of vertices}, where we also provide an example (\cref{exm: denominator 180}) showing that the upper bound of $180$ on the denominators of these polytopes is sharp. Thus, type-$\G$ BZ polytopes exhibit \emph{quasi-period collapse} for some values of $\lambda, \mu, \nu$, as in the classical types \cite{DeLMcA2006}.

In order to prove \cref{thm: possible denominators - intro version}, we describe an algorithm to compute the denominators of the vertices of all polyhedra in a parametrized family $\setof{\PP_t \st t \in \Znn^{k}}$, where $\PP_{t}$ is of the form
\begin{equation*}
    \PP_{t}
        \deftobe
        \setof{
            x \in \R^{n}
            \st
            Ax \le Bt,\,
            Cx = Dt
        } 
        \qquad \text{for all $t \in \Z^{k}$ with $t \ge 0$}.
\end{equation*}
SageMath code implementing this algorithm and applying it to the type-$\G$ BZ~polytopes in \cref{thm: explicit equations and inequalities for type G2} appears in Appendix~\ref{appendix: denominators of vertices code}.

\subsubsection*{Acknowledgments} 
Computational experiments were performed using passagemath \cite{passagemath}, a fork of the SageMath computer algebra system. In particular, Ehrhart quasi\-poly\-nomials were computed using LattE \cite{LattE2013} as accessed through passagemath. ChatGPT-6.1~Sol was used to search the literature, to proofread the manuscript, to develop the algorithm in \cref{sec: denominators of vertices}, and to write the SageMath code implementing this algorithm in Appendix~\ref{appendix: denominators of vertices code}. Claude~Opus~5.5 was used to produce \TikZ{} code for diagrams. However, all text in this paper was composed and verified by the authors.

\section{Background}
\label{sec: background}

Throughout this paper,~$\g$ is a complex semisimple Lie algebra of rank $r$, $\h$~is a Cartan subalgebra of~$\g$ with dual space~$\hd$, $\roots$~is the root system of~$\g$ in~$\hd$, and $\posroots \subset \roots$~is the set of positive roots corresponding to a fixed Borel subalgebra of~$\g$ containing~$\h$. We begin by recalling standard terminology and notation for Lie algebras. Readers are referred to \cite{Hum1972, FH1991} for two excellent (but very different) introductions to this theory.

For each root $\alpha \in \Phi$, we write \(\g_\alpha\) for the corresponding root space and~$\co{\alpha}$ for the corresponding coroot, which is defined to be the unique element \({\co{\alpha}\in\liebracket{\g_\alpha}{\g_{-\alpha}}}\) such that \(\alpha(\co{\alpha})=2\).  Further let \(e_\alpha \in \g_\alpha\) and \(f_\alpha \in \g_{-\alpha}\) be such that \(\liebracket{e_\alpha}{f_\alpha}=\co{\alpha}\).

Fix an ordering \(\alpha_{1}, \dotsc, \alpha_{r}\) of the simple roots. We write \(e_i \deftobe e_{\alpha_{i}} \) and \(f_i \deftobe f_{\alpha_{i}}\).  The \defing{Chevalley generators} are the elements \(e_i\), \(\co{\alpha_{i}}\), and~\(f_i\) for $i \in [r]$.
The \defing{Cartan matrix} is \(A=(a_{i,j})\), where \(a_{i,j}=\alpha_j(\alpha_i^\vee)\). The (integral) \defing{weights} of $\g$ are the elements \(\gamma \in \hd\) such that \(\gamma(\co{\alpha_i}) \in \Z\) for all $i \in [r]$. The weights constitute the \defing{weight lattice} $\weights$. The \defing{fundamental weights} $\omega_{1}, \dotsc, \omega_{r}$ are the basis of $\weights$ defined by \(\omega_j(\co{\alpha_i})=\delta_{i,j}\) for all $i, j \in [r]$, where \(\delta_{i,j}\) is the Kronecker delta. 
The \defing{Weyl cone} $\Wc$ is the nonnegative span of the fundamental weights. The \defing{dominant weights} are the weights in \(\Wc\). Indeed, the dominant weights are precisely the nonnegative \emph{integer} linear combinations of the fundamental weights.

The \defing{Langlands dual} of \(\g\) is the  Lie algebra \(\ld{\g}\) whose Cartan matrix is the transpose~\(\trans{A}\) of the Cartan matrix of $\g$.  This allows us to identify a Cartan subalgebra of \(\ld{\g}\) with \(\hd\), and the simple roots and coroots of \(\ld{\g}\) with the simple coroots and roots of \(\g\), respectively, under the same ordering. In particular, the fundamental weights of \(\ld{\g}\) are the fundamental coweights of~$\g$, defined by \(\alpha_j(\omega_i^\vee)=\delta_{i,j}\) for all \(i, j \in [r]\).

The \defing{Weyl group} of \(\g\) is the group \(\Wg\) generated by the set \(\setof{s_{i} \st i \in [r]}\) of simple reflections of \(\hd\), where \({s_i(\gamma)=\gamma-\gamma(\alpha_i^\vee)\alpha_i}\) for all \(\gamma \in \hd\). The action of $\Wg$ on $\h$ via \(s_i(h) \deftobe h-\alpha_i(h)\alpha_i^\vee\) for all \(h \in \h\) naturally identifies~$\Wg$ with the Weyl group of $\ld{\g}$.

Every finite-dimensional \g-module \(V\) has a \defing{weight-space de\-com\-po\-si\-tion}
\begin{equation*}
    V = \bigoplus_{\beta \in \weights} V(\beta),
\end{equation*}
in which \(V(\beta) \deftobe \setof{v \in V \st \text{$h \cdot v = \beta(h)v$ for all $h \in \h$} } \) is a \defing{weight space}.  Let \(W(V)\) be the set of weights of $V$, \emph{i.e.}, the set of weights \(\beta \in \weights\) such that \({V(\beta) \neq \{0\}}\). Up to isomorphism, the irreducible finite-dimensional \g-modules are indexed by the dominant weights $\lambda \in \Wc \cap \weights$, with~$V_{\lambda}$ being the unique such $\g$-module with highest weight $\lambda$. The corresponding \defing{weight polytope} is $P_{\lambda} \deftobe \conv\setof{w(\lambda) \st w \in \Wg}$ (where the convex hull is taken in the real span of~$\roots$), and $W(V_{\lambda})$ is the set of weights in $P_{\lambda}$ that are congruent to~$\lambda$ modulo the root lattice.

For all \(w \in \Wg\), a \defing{word} for \(w\) of length \(\ell\) is a sequence \((i_1, \dotsc, i_\ell) \in [r]^{\ell}\) such that \(w=s_{i_1}\cdots s_{i_\ell}\). The \defing{length} \(\ell(w)\) of \(w\) is the minimum length among all words for \(w\), and a \defing{reduced} word for $w$ is a word of length \(\ell(w)\). We write \(R(w)\) for the set of all reduced words for \(w\). Let \(w_0 \in \Wg\) be the unique element of $\Wg$ that is of maximum length among all elements of \(\Wg \).

Every triple $(\lambda, \mu, \nu)$ of dominant weights of $\g$ determines a tensor-product multiplicity (also called a \defing{Clebsch--Gordan coefficient}) \(\clnm\), which is the multiplicity of \(V_\mu\) in the tensor product \(V_\lambda \otimes V_\nu\). That is, for all $\lambda, \nu \in \Wc \cap \weights$,
\begin{equation*}
    V_\lambda \otimes V_\nu 
        \isom
        \bigoplus_{\mu\in\Wc \cap \weights} V_{\mu}^{\oplus \clnm} .
\end{equation*}
Berenstein and Zelevinsky define two families of rational polytopes in $\R^{\ell(w_{0})}$, both of which contain a number of integer-lattice points equal to the multiplicity \(\clnm\) of \(V_\mu\) in \(V_\lambda \otimes V_\nu\) (\cite[Section~2.2]{BerZel2001}; see also \cref{defn: c-vec polytope,defn: d-vec polytope} below).  These polytopes are presented in terms of certain so-called \emph{\(\i\)-trails} and their corresponding \emph{c-} and \emph{d-vectors}, which we now define.

\begin{definition} \label{defn: potential i-trail}
    Let \(V\) be a finite-dimensional \g-module. Fix $\ell \in \Zp$ and \begin{equation*}
        {\i = (i_1,\dotsc, i_\ell) \in [r]^{\ell}}.
    \end{equation*} 
    A \defing{potential $\i$-trail} in $V$ is a sequence \(\pi = (\gamma_0,\gamma_1,\cdots, \gamma_\ell)\) of weights in \(W(V)\) such that, for all \(k \in [\ell]\), there exists a nonnegative integer \(c_k\) such that \(\gamma_{k-1}-\gamma_k = c_k \alpha_{i_k}\). In this case we say that $\pi$ is a potential $\i$-trail \defing{from~$\gamma_{0}$ to~$\gamma_{\ell}$}, and that the \defing{c-vector} of $\pi$ is $c(\pi) \deftobe (c_{1}, \dotsc, c_{\ell})$.
\end{definition} 

Because of the possibility of zero entries in the c-vector, a potential $\i$-trail may contain a \defing{pause}, \emph{i.e.}, a weight that appears in two or more consecutive entries of the potential $\i$-trail.

\begin{example}\label{i-trail example}
    Let $\g$ be of type $\G$, and order the simple roots of $\g$ according to the Bourbaki convention, so that $\alpha_{1}$ is the short simple root of $\g$ and $\alpha_{2}$ is the long simple root of $\g$. Thus, $\alpha_{1}^{\vee}$ is the \emph{long} simple root of $\ld\g$ and $\alpha_{2}^{\vee}$ is the \emph{short} simple root of $\ld\g$. Let $\i = (1,2,1,2,1,2)$ and let \(V \deftobe V_{\omega_{1}^{\vee}}\) be the adjoint representation of $\ld\g$. Label the weights of $V$ as follows:
    
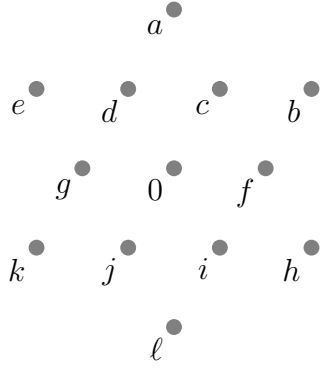
\begin{center}
        \begin{tikzpicture}[scale=0.7]
            \begin{scope}[x=(30:1)]
                \draw ( 0,  0) node[graypoint]{}
                    node[below left] {$0$};
                \draw ( 2, -1) node[graypoint]{}
                    node[below left] {$f$};
                \draw (-3,  3) node[graypoint]{}
                    node[below left] {\(e\)};
                \draw (-1,  2) node[graypoint]{}
                    node[below left] {\(d\)};
                \draw ( 1,  1) node[graypoint]{}
                    node[below left] {\(c\)};
                \draw ( 3,  0) node[graypoint]{}
                    node[below left] {\(b\)};
                \draw ( 0,  3) node[graypoint]{}
                    node[below left] {\(a\)};
                \draw (-2,  1) node[graypoint]{}
                    node[below left] {\(g\)};
                \draw ( 3, -3) node[graypoint]{}
                    node[below left] {\(h\)};
                \draw ( 1, -2) node[graypoint]{}
                    node[below left] {\(i\)};
                \draw (-1, -1) node[graypoint]{}
                    node[below left] {\(j\)};
                \draw (-3,  0) node[graypoint]{}
                    node[below left] {\(k\)};
                \draw ( 0, -3) node[graypoint]{}
                    node[below left] {\(\ell\)};
            \end{scope}
        \end{tikzpicture}
    \end{center}
    In particular, $a = \omega_{1}^{\vee}$, $e = \alpha_{1}^{\vee}$, and $f = \alpha_{2}^{\vee}$. See \cref{fig: four example i-trails} for four examples of potential $\i$-trails $\pi_{1}, \pi_{2}, \pi_{3}, \pi_{4}$ from $a$ to $k$.
    \begin{figure}
        \begin{subfigure}{0.48\textwidth}
            \centering
            \begin{tikzpicture}[scale=0.7]
                \begin{scope}[x=(30:1)]
                    \draw ( 0,  0) node[graypoint]{}
                        node[below left] {$0$};
                    \draw ( 2, -1) node[graypoint]{}
                        node[below left] {$f$};
                    \draw (-3,  3) node[redpoint]{}
                        node[below left] {\(e\)};
                    \draw (-1,  2) node[redpoint]{}
                        node[below left] {\(d\)};
                    \draw[red] (-1,2) circle (.3cm);
                    \draw ( 1,  1) node[graypoint]{}
                        node[below left] {\(c\)};
                    \draw ( 3,  0) node[redpoint]{}
                        node[below left] {\(b\)};
                    \draw ( 0,  3) node[redpoint]{}
                        node[below left] {\(a\)};
                    \draw (-2,  1) node[graypoint]{}
                        node[below left] {\(g\)};
                    \draw ( 3, -3) node[redpoint]{}
                        node[below left] {\(h\)};
                    \draw ( 1, -2) node[graypoint]{}
                        node[below left] {\(i\)};
                    \draw (-1, -1) node[graypoint]{}
                        node[below left] {\(j\)};
                    \draw (-3,  0) node[redpoint]{}
                        node[below left] {\(k\)};
                    \draw ( 0, -3) node[graypoint]{}
                        node[below left] {\(\ell\)};
                    \draw[red,thick] (0,3)--(3,0)--(-3,3)--(3,-3)--(-3,0);
                \end{scope}
            \end{tikzpicture}
            $\pi_1 \deftobe (a,b,d,d,e,h,k)$ \\
            $c(\pi_{1}) = (1,2,0,1,2,3)$ \\
            path of $\pi_{1}$: $(a,b,e,h,k)$
        \end{subfigure}\hfill
        \begin{subfigure}{0.48\textwidth}
            \centering
            \begin{tikzpicture}[scale=0.7]
                \begin{scope}[x=(30:1)]
                    \draw ( 0,  0) node[graypoint]{}
                        node[below left] {$0$};
                    \draw ( 2, -1) node[graypoint]{}
                        node[below left] {$f$};
                    \draw (-3,  3) node[redpoint]{}
                        node[below left] {\(e\)};
                    \draw (-1,  2) node[graypoint]{}
                        node[below left] {\(d\)};
                    \draw ( 1,  1) node[graypoint]{}
                        node[below left] {\(c\)};
                    \draw ( 3,  0) node[redpoint]{}
                        node[below left] {\(b\)};
                    \draw ( 0,  3) node[redpoint]{}
                        node[below left] {\(a\)};
                    \draw (-2,  1) node[graypoint]{}
                        node[below left] {\(g\)};
                    \draw ( 3, -3) node[redpoint]{}
                        node[below left] {\(h\)};
                    \draw ( 1, -2) node[redpoint]{}
                        node[below left] {\(i\)};
                    \draw[red] (1,-2) circle (.3cm);
                    \draw (-1, -1) node[graypoint]{}
                        node[below left] {\(j\)};
                    \draw (-3,  0) node[redpoint]{}
                        node[below left] {\(k\)};
                    \draw ( 0, -3) node[graypoint]{}
                        node[below left] {\(\ell\)};
                    \draw[red,thick] (0,3)--(3,0)--(-3,3)--(3,-3)--(-3,0);
                \end{scope}
            \end{tikzpicture}
            $\pi_2 \deftobe (a,b,e,h,i,i,k)$ \\
            $c(\pi_{2}) = (1,3,2,1,0,2)$ \\
            path of $\pi_{2}$: $(a, b, e, h, k)$
        \end{subfigure}

        \bigskip
        \bigskip

        \begin{subfigure}{0.48\textwidth}
            \centering
            \begin{tikzpicture}[scale=0.7]
                \begin{scope}[x=(30:1)]
                    \draw ( 0,  0) node[graypoint]{}
                        node[below left] {$0$};
                    \draw ( 2, -1) node[redpoint]{}
                        node[below left] {$f$};
                    \draw (-3,  3) node[graypoint]{}
                        node[below left] {\(e\)};
                    \draw (-1,  2) node[redpoint]{}
                        node[below left] {\(d\)};
                    \draw ( 1,  1) node[graypoint]{}
                        node[below left] {\(c\)};
                    \draw ( 3,  0) node[redpoint]{}
                        node[below left] {\(b\)};
                    \draw ( 0,  3) node[redpoint]{}
                        node[below left] {\(a\)};
                    \draw (-2,  1) node[redpoint]{}
                        node[below left] {\(g\)};
                    \draw ( 3, -3) node[graypoint]{}
                        node[below left] {\(h\)};
                    \draw ( 1, -2) node[redpoint]{}
                        node[below left] {\(i\)};
                    \draw (-1, -1) node[graypoint]{}
                        node[below left] {\(j\)};
                    \draw (-3,  0) node[redpoint]{}
                        node[below left] {\(k\)};
                    \draw ( 0, -3) node[graypoint]{}
                        node[below left] {\(\ell\)};
                    \draw[red,thick] (0,3)--(3,0)--(-1,2)--(2,-1)--(-2,1)--(1,-2)--(-3,0);
                \end{scope}
            \end{tikzpicture}
            $\pi_3 \deftobe (a,b,d,f,g,i,k)$ \\
            $c(\pi_{3}) = (1,2,1,2,1,2)$ \\
            path of $\pi_{3}$: $(a,b,d,f,g,i,k)$
        \end{subfigure}\hfill
        \begin{subfigure}{0.48\textwidth}
            \centering
            \begin{tikzpicture}[scale=0.7]
                \begin{scope}[x=(30:1)]
                    \draw ( 0,  0) node[redpoint]{}
                        node[below left] {$0$};
                    \draw ( 2, -1) node[redpoint]{}
                        node[below left] {$f$};
                    \draw (-3,  3) node[graypoint]{}
                        node[below left] {\(e\)};
                    \draw (-1,  2) node[redpoint]{}
                        node[below left] {\(d\)};
                    \draw ( 1,  1) node[graypoint]{}
                        node[below left] {\(c\)};
                    \draw ( 3,  0) node[redpoint]{}
                        node[below left] {\(b\)};
                    \draw ( 0,  3) node[redpoint]{}
                        node[below left] {\(a\)};
                    \draw (-2,  1) node[graypoint]{}
                        node[below left] {\(g\)};
                    \draw ( 3, -3) node[redpoint]{}
                        node[below left] {\(h\)};
                    \draw ( 1, -2) node[graypoint]{}
                        node[below left] {\(i\)};
                    \draw (-1, -1) node[graypoint]{}
                        node[below left] {\(j\)};
                    \draw (-3,  0) node[redpoint]{}
                        node[below left] {\(k\)};
                    \draw ( 0, -3) node[graypoint]{}
                        node[below left] {\(\ell\)};
                    \draw[red,thick] (0,3)--(3,0)--(-1,2)--(2,-1)--(0,0)--(3,-3)--(-3,0);
                \end{scope}
            \end{tikzpicture}
            $\pi_4 \deftobe (a,b,d,f,0,h,k)$ \\
            $c(\pi_{4}) = (1,2,1,1,1,3)$ \\
            path of $\pi_{4}$: $(a,b,d,f,0,h,k)$
        \label{fig:d}
        \end{subfigure}
        \caption{%
            Four potential $\i$-trails $\pi_{1}, \pi_{2}, \pi_{3}, \pi_{4}$ in the adjoint representation in \cref{i-trail example}, with $\i = (1,2,1,2,1,2)$. The red dots indicate the weights that appear in the potential $\i$-trail. Circled weights indicate where the potential $\i$-trail pauses.
        }
        \label{fig: four example i-trails}
    \end{figure}
    Observe that the potential $\i$\nobreakdash-trails~$\pi_1$ and $\pi_2$ each contain a pause. We have indicated these pauses by circling the repeated weight.  
\end{example}

The potential $\i$-trails $\pi_{1}$ and $\pi_{2}$ in \cref{fig: four example i-trails} are different, but the ``paths traced'' by these potential $\i$-trails are the same. Roughly speaking, we define the \emph{path} of a potential $\i$-trail ${\pi = (\gamma_{0}, \gamma_{1}, \dotsc, \gamma_{\ell})}$ to be the unique subsequence of weights in $\pi$ that begins with $\gamma_{0}$, ends with $\gamma_{\ell}$, and includes only the ``intermediate'' weights where $\pi$ ``changes direction'', with each weight in the path appearing only once. We formalize this notion of ``path'' in \cref{defn:path of an i-trail} below.

\begin{definition}\label{defn:path of an i-trail}
    Let $\i = (i_{1}, \dotsc, i_{\ell}) \in [r]^{\ell}$ and let \(\pi = (\gamma_{0}, \dotsc, \gamma_{\ell})\) be a potential $\i$-trail, with c-vector $(c_{1}, \dotsc, c_{\ell})$. 
    Let $I \deftobe \setof{k \in [\ell] \st c_{k} = 0}$ be the set of indices where the c-vector is zero. 
    Let
    $c' = (c'_{1}, \dotsc, c'_{\ell'})$, 
    respectively $\i' = (i'_{1}, \dotsc, i'_{\ell'})$, 
    be the result of removing the entries with indices in $I$ from~$c$, respectively~$\i$. 
    Partition $c'$ into subwords of maximal length subject to the condition that, in the $j$-th subword $(c'_{k_{j}}, \dotsc, c'_{k_{j+1} - 1})$ (beginning with $j = 1$ and setting $k_{1} = 1$ and $k_{\ell'' + 1} = \ell' + 1$), the corresponding entries of $\i'$ are equal: $i'_{k_{j}} = \dotsb = i'_{k_{j+1} - 1}$, and let $\ell''$ be the number of such subwords. Let $c'' = (c''_{1}, \dotsc, c''_{\ell''})$ be the result of replacing each such subword of $c'$ by its sum:
    \begin{equation*}
        c''_{j} \deftobe \sum_{s = k_{j}}^{k_{j+1} - 1} c'_{s},
        \qquad
        \text{for $1 \le j \le \ell''$},
    \end{equation*}
    and let $\i'' \deftobe(i''_{1}, \dotsc, i''_{\ell''})$ with $i''_{j} \deftobe i'_{k_{j}}$ for $1 \le j \le \ell''$.
    Finally, define the \defing{path} $(\delta_{0}, \delta_{1}, \dotsc, \delta_{\ell''})$ of $\pi$ by setting $\delta_{0} = \gamma_{0}$ and $\delta_{j-1} - \delta_{j} = c''_{j} \alpha_{i''_{j}}$ for $1 \le j \le \ell''$.
\end{definition}

\begin{definition}
    We call two potential $\i$-trails $\pi$ and $\pi'$ \defing{path equivalent} if $\pi$ and $\pi'$ have the same path.
\end{definition}

\begin{example}
    See again \cref{fig: four example i-trails} for the paths of the potential $\i$-trails $\pi_{1}, \pi_{2}, \pi_{3}, \pi_{4}$ from \cref{i-trail example}. The potential $\i$-trails $\pi_{1}$ and $\pi_{2}$ are path equivalent.
\end{example}

The conditions defining \emph{potential} $\i$-trails are purely combinatorial. In order to distinguish the ``actual'' $\i$-trails from the merely potential $\i$-trails, Berenstein and Zelevinsky impose an additional non-combinatorial condition that we call the \emph{composition rule}.

\begin{definition} \label{defn: i-trail}
    Let \(V\) be a finite-dimensional \g-module, let $\gamma, \delta \in W(V)$, and let $\i = (i_{1}, \dotsc, i_{\ell}) \in [r]^{\ell}$. An \defing{$\i$-trail from $\gamma$ to $\delta$ in $V$} is a potential $\i$-trail $\pi$ from $\gamma$ to $\delta$ in $V$ such that the \defing{path operator} \[p_\pi = e^{c_1}_{i_1}\cdots e^{c_\ell}_{i_\ell}\] is a non-zero linear map from \(V(\delta) \) to \(V(\gamma)\), where $(c_{1}, \dotsc, c_{\ell})$ is the c\nobreakdash-vector of $\pi$.  
    We call this non-zero requirement on the path operator the \defing{composition rule}.
\end{definition}

The path operator of a potential $\i$-trail is determined by its path. In particular, whether a potential $\i$-trail is an (actual) $\i$-trail depends only on its path.

\begin{lemma}\label{path operator lemma}
    If $\pi$ and $\pi'$ are path-equivalent potential $\i$-trails, then $p_{\pi} = p_{\pi'}$. Thus, if $\pi$ is moreover an $\i$-trail, then so is $\pi'$.
\end{lemma}

\begin{example}
    In \cref{fig: four example i-trails}, we can write the path operators as follows, confirming that \(\pi_1\) and \(\pi_2\) have the same path operator:
    \begin{align*}
        p_{\pi_1} 
            &=
            e_{1}^{1} e_{2}^{2} e_{1}^{0} e_{2}^{1} e_{1}^{2} e_{2}^{3}
            =
            e_{1}^{1} e_{2}^{3} e_{1}^{2} e_{2}^{3}, \\
        p_{\pi_2} 
            & =
            e_{1}^{1} e_{2}^{3} e_{1}^{2} e_{2}^{1} e_{1}^{0} e_{2}^{2}
            =
            e_{1}^{1} e_{2}^{3} e_{1}^{2} e_{2}^{3}, \\
        p_{\pi_3} 
            & =
            e_{1}^{1} e_{2}^{2} e_{1}^{1} e_{2}^{2} e_{1}^{1} e_{2}^{2}, \\
        p_{\pi_4} 
            & =
            e_{1}^{1} e_{2}^{2} e_{1}^{1} e_{2}^{1} e_{1}^{1} e_{2}^{1}.
    \end{align*}
\end{example}

Note that the c-vector $c(\pi) = \bigparens{c_{1}(\pi), \dotsc, c_{\ell}(\pi)}$ of a potential \(\i\)-trail \(\pi\) may also be written as follows:
\begin{equation*}
    c_k(\pi)=\frac{\gamma_{k-1} - \gamma_k}{2}(\alpha^{\vee}_{i_k}).
\end{equation*}
Continuing to follow \cite{BerZel2001}, we define the d-vector of $\pi$ as follows:
\begin{definition}
    Let $\i = (i_{1}, \dotsc, i_{\ell}) \in [r]^{\ell}$ and let $\pi$ be a potential $\i$-trail. The \textbf{d-vector} \(d(\pi)=\bigparens{d_1(\pi), \dotsc, d_\ell(\pi)}\) of \(\pi\) is defined by 
    \begin{equation*}
        d_k(\pi)=\frac{\gamma_{k-1} + \gamma_k}{2}(\alpha^{\vee}_{i_k}).
    \end{equation*}
\end{definition}

In \cite[Section 2.3]{BerZel2001}, Berenstein and Zelevinsky use the~c-~and d-vectors of $\i$-trails to define two families of polytopes, each of which they show contains a number of lattice points equal to the Clebsch--Gordan coefficient $\clnm$ corresponding to a given triple of dominant weights of a semisimple Lie algebra. We reproduce the definitions of these polytopes here.

\begin{definition}\label{defn: c-vec polytope}
    Fix a reduced word \( \i = (i_1, \dotsc, i_\ell) \in R(w_0)\).  Let \(\lambda, \mu, \nu\) be three dominant weights of the semisimple Lie algebra~\(\mathfrak{g}\). 
    The polytope $\bzlmnc$ is the set of all solutions $(t_{1}, \dotsc, t_{\ell}) \in \R^{\ell}$ to the following system of linear equations and inequalities:
    \begin{defnconditions}
    	\item
    	\(t_k \geq 0\) for all \(k\in [\ell]\),
    
    	\item
    	\(\sum_{k=1}^\ell t_ks_{i_1}\cdots s_{i_{k-1}}\alpha_{i_k}=\lambda+\nu-\mu\),
    
    	\item
    	\(\sum_{k=1}^{\ell} c_k(\pi)t_k\geq (s_i \lambda+\nu-\mu)(\omega_i^\vee)\) for all $i \in [r]$ and $\i$-trails \(\pi\) from \(s_i\omega_i^\vee\) to \(w_0\omega_i^\vee\) in~\(V_{\omega_i^\vee}\),
    
    	\item
    	\(\sum_{k=1}^{\ell}c_k(\pi)t_k\geq (\lambda+ s_i \nu-\mu)(\omega_i^\vee)\) for all $i \in [r]$ and $\i$-trails \(\pi\) from \(\omega_i^\vee\) to \(w_0s_i\omega_i^\vee\) in~\(V_{\omega_i^\vee}\).
    \end{defnconditions}
\end{definition}

\begin{definition}\label{defn: d-vec polytope}
    Fix a reduced word \(\i = (i_1, \dotsc, i_\ell) \in R(w_0)\). Let \(\lambda, \mu, \nu\) be three dominant weights of the semisimple Lie algebra~\(\mathfrak{g}\). 
    The polytope $\bzlmnd$ is the set of all solutions $(t_{1}, \dotsc, t_{\ell}) \in \R^{\ell}$ to the following system of linear equations and inequalities:
    \begin{defnconditions}
        \item 
        \(\sum_{k=1}^{\ell} d_k(\pi) t_k \geq 0\) for all $i \in [r]$ and $\i$-trails \(\pi\) from \(\omega_i^\vee\) to \(w_0s_i\omega_i^\vee\) in \(V_{\omega_i^\vee}\),
        
        \item
        \(\sum_{k=1}^{\ell} t_k\alpha_{i_k}=\lambda+\nu-\mu\),
        
        \item \(\sum_{k=1}^{\ell} d_k(\pi)t_k\geq-\lambda(\alpha_i^\vee)\) for all $i \in [r]$ and $\i$-trails \(\pi\) from \(s_i\omega_i^\vee\) to \(w_0\omega_i^\vee\) in \(V_{\omega_i^\vee}\),
        
        \item
        \(t_k + \sum_{l=k+1}^{\ell}a_{i_k,i_l}t_l\leq \nu(\alpha^\vee_{i_k})\) for all $k \in [\ell]$.
    \end{defnconditions}
\end{definition}

\begin{theorem}\label{thm: bz-polytopes}\cite[Theorems~2.3 and~2.4]{BerZel2001}
	Let \(\lambda, \mu, \nu\) be dominant weights of the semisimple Lie algebra~\(\mathfrak{g}\). Then, for every choice of reduced word $\i = (i_{1}, \dotsc, i_{\ell}) \in R(w_{0})$, the number of integer-lattice points in $\bzlmnc$, respectively $\bzlmnd$, is exactly the tensor-product multiplicity~$\clnm$:
    \begin{equation*}
        \card{\bzlmnc \cap \Z^{\ell}} 
            = 
            \clnm
            = 
            \card{\bzlmnd \cap \Z^{\ell}}.
    \end{equation*}
\end{theorem}

\section{BZ polytopes for type \texorpdfstring{$\G$}{G\unisubtwo}}

We continue to let $\g$ be a finite-dimensional complex semisimple Lie algebra. The goal of this section is to write completely explicit systems of linear equations and inequalities for the Berenstein--Zelevinsky polytopes $\bzlmnc$ in the type-$\G$ case. Our discussion in \cref{subsec: potential i-trails as lattice points} will apply to the general semisimple case, with the type-$\G$ case developed as a running example. In 
\cref{%
    subsec: Potential i-trails for type-G2 BZ polytopes,%
    subsec: Triviality of the composition rule in type G2,%
    subsec: Explicit linear equations and inequalities for type-G2 BZ polytopes%
}, we will restrict our attention to the type-$\G$ case.

Throughout the remainder of this section, let $\i = (i_{1}, \dotsc, i_{\ell})$ be a fixed reduced word for the longest element \(w_0 \in \Wg\). 
\begin{example}
    In our running example, we let $\g$ be  of type \(\G\) and $\ell \deftobe \ell(w_{0}) = 6$. In anticipation of the construction of the BZ polytopes, we will take $\i$ to be a reduced word for~$w_{0}$, of which there are two, namely \(s_1s_2s_1s_2s_1s_2\) and \(s_2s_1s_2s_1s_2s_1\).  We choose \(s_1s_2s_1s_2s_1s_2\) and thus \(\i \deftobe (1,2,1,2,1,2)\).
\end{example}
To compute explicit expressions for the BZ polytopes, we first need to determine the potential $\i$-trails in a given irreducible $\g$-module $V_{\lambda}$. Then we will confirm that, in type~\(\G\), all the potential $\i$-trails that arise in \cref{defn: c-vec polytope,defn: d-vec polytope} satisfy the composition rule in the definition of $\i$-trails (\cref{defn: i-trail}) and so are in fact \(\i\)-trails, proving \cref{thm: in G2 every relevant potential i-trail is an actual i-trail}.

\subsection{c-vectors of potential \texorpdfstring{$\i$}{\unibi}-trails}
\label{subsec: potential i-trails as lattice points}

Fix a dominant weight \(\lambda\) of $\g$ and weights $\gamma, \delta \in W(V_{\lambda})$. The conditions in \cref{defn: potential i-trail} on the potential \(\i\)-trails from~\(\gamma\) to~\(\delta\) in~\(V_{\lambda}\) determine a system of linear equations and inequalities that define a polytope \(\pip \subset \R^{\ell}\) whose lattice points are precisely the c-vectors of these potential \(\i\)-trails. 

In this subsection, we write explicit linear equations and inequalities for $\pip$. When the starting point $\gamma$ is fixed, a c-vector determines its potential $\i$-trail. Thus, the polytope $\pip$ defined below encodes all potential $\i$-trails from $\gamma$ to $\delta$ in $V_{\lambda}$.

Recall the conditions on a potential $\i$-trail \(\pi = (\gamma=\gamma_0,\gamma_1,\cdots, \gamma_{\ell}=\delta )\) in \(W(V_{\lambda})\):
\begin{enumerate}
    \item 
    \(\pi\) starts at \(\gamma\) and ends at \(\delta\).

    \item 
    The components of the c-vector \(c(\pi)\) are nonnegative integers.
    
    \item 
    Each weight in \(\pi\) is in \(W(V_{\lambda})\).
    
\end{enumerate}

We now express these conditions as a system of linear equations and inequalities on the c-vectors of potential $\i$-trails from $\gamma$ to $\delta$ in $V_{\lambda}$.

\begin{definition}
    Let \(T_\i\) be the \(r \times \ell\) matrix whose \(j\)-th column is the $i_{j}$-th standard-basis vector in $\R^{r}$ for $1 \le j \le \ell$.
\end{definition}

\begin{example}\label{exm: Ti for G2}
    For the type-$\G$ case, we have \(\ell = \ell(w_{0}) = 6\) and \({\rank(\g) = 2}\), and we have fixed \(\i = (1,2,1,2,1,2)\). Thus,
    \begin{equation*}
        T_\i = \mqty[ 1&0&1&0&1&0\\0&1&0&1&0&1].
    \end{equation*}
\end{example}

Every potential $\i$-trail in $V_{\lambda}$ from $\gamma$ to $\delta$ with c-vector $c = (c_{1}, \dotsc, c_{\ell})$ begins at \(\gamma\), takes steps $-c_{k}\alpha_{i_{k}}$ for $1 \le k \le \ell$ (in that order), and ends at \(\delta\).  Thus the c-vector of such a potential $\i$-trail satisfies
\begin{equation}\label{c-vec equality}
    T_\i c = [\gamma-\delta],
\end{equation}
where we write $[\eta]$ for the column vector of a point $\eta \in \rf$ with respect to the ordered basis $\alpha_{1}, \dotsc, \alpha_{r}$ of simple roots.
This is the only equality condition that the c-vectors of these potential $\i$-trails are required to satisfy.

We now impose the inequality conditions on c-vectors corresponding to the requirement that each weight $\gamma_{k}$ of a potential $\i$-trail $(\gamma = \gamma_{0}, \gamma_{1}, \dotsc, \gamma_{\ell})$ is in $W(V_{\lambda})$. Recall from \cref{sec: background} that $W(V_{\lambda})$ is the set of points in the weight polytope $P_{\lambda}$ that are congruent to~$\lambda$ modulo the root lattice. The map $\eta \mapsto [\eta]$ identifies the real span of $\Phi$ in $\hd$ with~$\R^r$. Thus, $[P_{\lambda}] \deftobe \setof{[\eta] \st \eta \in P_{\lambda}} = \setof{x \in \R^r \st A_{\lambda}x \leq b_{\lambda}}$ for some matrix \(A_\lambda\) and vector~\(b_\lambda\), and we require 
\begin{equation}\label{c-vec repn poly}
    A_\lambda [\gamma_{k}] \leq b_\lambda.
\end{equation}
Now, given a c-vector~$c$, the weight $\gamma_{k}$ of the corresponding potential $\i$-trail is given by 
\begin{equation}
    [\gamma_{k}]
        = 
        [\gamma]
        - 
        T_\i
        \mqty[\dmat{\epsilon_{k}(1),
        \ddots,\epsilon_{k}(\ell)}]c
\end{equation}
where $\epsilon_{k}(1) = \dotsb =\epsilon_{k}(k) = 1$ and \(\epsilon_{k}(k+1) = \dotsb = \epsilon_{k}(\ell) = 0\), for all \(k \in [\ell]\). Thus the corresponding inequality constraints on the c-vectors are
\begin{equation*}
    A_{\lambda} 
    T_{\i}
    \mqty[\dmat{\epsilon_k(1),
    \ddots,\epsilon_{k}(\ell)}]
    c
        \ge
        A_{\lambda} [\gamma]
        -
        b_{\lambda},
\end{equation*}
for all $k \in [\ell]$. (The additional congruence requirement on the weights in $W(V_{\lambda})$ is automatic, since each step is an integer multiple of a simple root.)

We summarize the foregoing in the following proposition describing the polytope $\pip$.

\begin{proposition}
    The c-vectors of potential $\i$-trails from $\gamma$ to $\delta$ in $V_{\lambda}$ are precisely the integer-lattice points in the rational polytope $\pip$ consisting of the points $c \in \R^{\ell}$ such that
    \begin{equation}\label{system: PIP}
        \begin{split}
            T_\i c 
                & \;=\; 
                [\gamma-\delta], \\[2ex]
            A_{\lambda}
            T_{\i}
            \mqty[\dmat{\epsilon_k(1),
            \ddots,\epsilon_{k}(\ell)}]
            c
                & \;\ge\;
                A_{\lambda} [\gamma]
                -
                b_{\lambda} \quad \text{for all $k \in [\ell]$}, \\[2ex]
            c
                & \ge
                0,
        \end{split}
    \end{equation}
    where $T_{\i}$, $A_{\lambda}$, $b_{\lambda}$, $[\eta]$, and $\epsilon_{k}(j)$ are as defined above.
\end{proposition}

\begin{example}\label{exm: type G2 pip}
    Applying the above to the type-$\G$ case, we find that the c-vectors of potential $\i$-trails from $\gamma$ to $\delta$ in $V_{\lambda}$ are the lattice points in the polytope $\pip$ defined by  
    \begin{equation}\label{g2 complete itrail polytope}
        \begin{split}
            \mqty[ 1&0&1&0&1&0\\0&1&0&1&0&1]c &= [\gamma-\delta], \\
            A_\lambda\mqty[ 1&0&0&0&0&0\\0&0&0&0&0&0]c
                & \geq A_{\lambda}[\gamma]-b_\lambda, \\
            A_\lambda\mqty[ 1&0&0&0&0&0\\0&1&0&0&0&0]c
                & \geq A_{\lambda}[\gamma]-b_\lambda, \\
            A_\lambda\mqty[ 1&0&1&0&0&0\\0&1&0&0&0&0]c
                & \geq A_{\lambda}[\gamma]-b_\lambda, \\
            A_\lambda\mqty[ 1&0&1&0&0&0\\0&1&0&1&0&0]c
                & \geq A_{\lambda}[\gamma]-b_\lambda, \\
            A_\lambda\mqty[ 1&0&1&0&1&0\\0&1&0&1&0&0]c
                & \geq A_{\lambda}[\gamma]-b_\lambda, \\
            A_\lambda\mqty[ 1&0&1&0&1&0\\0&1&0&1&0&1]c
                & \geq A_{\lambda}[\gamma]-b_\lambda, \\
            c & \ge 0.
        \end{split}
    \end{equation}
\end{example}

\subsection{Potential \texorpdfstring{$\i$}{\unibi}-trails for type-\texorpdfstring{$\G$}{G\unisubtwo} BZ polytopes}
\label{subsec: Potential i-trails for type-G2 BZ polytopes}

We now continue \cref{exm: type G2 pip}.  Furthermore, we restrict to the case in which \i\ is a reduced word for \(w_0\), \(\lambda\) is a fundamental weight, and \(\gamma\) and \(\delta\) are the starting and stopping points defined in \cref{defn: c-vec polytope}.
From this point forward, we assume that $\g$ is of type~$\G$, with the simple roots ordered so that $\alpha_{1}$ (respectively~$\alpha_{2}$) is the short (respectively~long) simple root. Thus, the Cartan matrix is
\begin{equation*}
    A
        = 
        \begin{bmatrix*}[r]
             2 & -3 \\
            -1 &  2  
        \end{bmatrix*}.
\end{equation*}
The Langlands dual \(\ld{\g}\) of $\g$ is also of type~\(\G\), but with Cartan matrix $\trans{A}$. Thus, the simple roots of \(\ld{\g}\) are $\co{\alpha_{1}}$ and $\co{\alpha_{2}}$, where $\co{\alpha_{1}}$ (respectively~\(\co{\alpha_{2}}\)) is the \emph{long} (respectively~\emph{short}) simple root of $\G$.

We continue to use the notation from the discussion of the type-$\G$ case in \cref{subsec: potential i-trails as lattice points}. Thus, \(\i = (1,2,1,2,1,2)\). 

The matrices $A_{\omega_{i}^{\vee}}$ and $b_{\omega_{i}^{\vee}}$ defining the weight polytopes $P_{\omega_{i}^{\vee}}$ for $i = 1, 2$ are as follows:
\begin{align*}
    [P_{\co\omega_1}] &= \setof{\mqty[x_1\\x_2]\in\R^2 \st \begin{bmatrix*}[r]
        -3&2\\0&-1\\-3&1\\3&-2\\3&-1\\0&1
    \end{bmatrix*}\mqty[x_1\\x_2]\leq\mqty[3\\3\\3\\3\\3\\3]},
    \\
    [P_{\co\omega_2}] &= \setof{\mqty[x_1\\x_2]\in\R^2 \st \begin{bmatrix*}[r]
        -2&1\\1&-1\\-1&0\\2&-1\\1&0\\-1&1
    \end{bmatrix*}\mqty[x_1\\x_2]\leq\mqty[1\\1\\1\\1\\1\\1]}.
\end{align*}
The entries in the following table are the \(c\)-vectors of the potential $\i$-trails arising in \cref{defn: c-vec polytope}.

\begin{table}[htbp]
  \centering
  \renewcommand{\arraystretch}{1.25}
  \begin{tabular}{|c|c|c|c|}
    \hline
    Condition 3, $\co\omega_1$ & Condition 3, $\co\omega_2$ & Condition 4, $\co\omega_1$ & Condition 4, $\co\omega_2$ \\[0.5ex]
    \hline
    (0, 3, 2, 3, 1, 0) & (0, 0, 1, 2, 1, 1) & (0, 0, 1, 3, 2, 3) & (0, 1, 1, 2, 1, 0) \\
                       & (1, 0, 0, 2, 1, 1) & (1, 0, 0, 3, 2, 3) & \\
                       & (1, 1, 0, 1, 1, 1) & (1, 1, 0, 2, 2, 3) & \\
                       & (1, 2, 0, 0, 1, 1) & (1, 2, 0, 1, 2, 3) & \\
                       & (1, 2, 1, 1, 0, 0) & (1, 3, 0, 0, 2, 3) & \\
                       & (1, 2, 1, 0, 0, 1) & (1, 2, 1, 1, 1, 3) & \\
                       &                    & (1, 2, 1, 2, 1, 2) & \\
                       &                    & (1, 3, 2, 0, 0, 3) & \\
                       &                    & (1, 3, 1, 1, 1, 2) & \\
                       &                    & (1, 3, 1, 0, 1, 3) & \\
                       &                    & (1, 3, 2, 1, 0, 2) & \\
                       &                    & (1, 3, 2, 2, 0, 1) & \\
                       &                    & (1, 3, 2, 3, 0, 0) & \\
    \hline
  \end{tabular}
\end{table}

By \cref{path operator lemma}, to decide whether a potential $\i$-trail is an $\i$-trail, we need only consider the potential $\i$-trail up to path
equivalence. Using the notation of \cref{i-trail example}, we list below the distinct paths and their path operators.

\subsection{\texorpdfstring{Condition (3) of \cref{defn: c-vec polytope}: trails from $s_i\omega_i^\vee$ to $w_0\omega_i^\vee$}{Condition (3)}}
\subsubsection{$V_{\omega_1^\vee}$, maps $V((-2,-3))\to V((1,3))$: 1 potential $\i$-trail, 1 path}
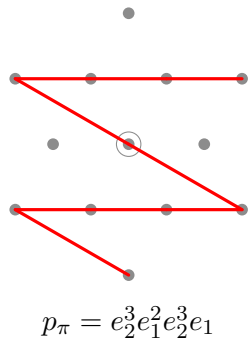
\begin{center}
\begin{minipage}[t]{0.32\textwidth}\centering
\begin{tikzpicture}[scale=1.0]\label{fig: 14dim it paths 1}
\fill[wt] (0,1.732) circle (2.20pt);
\fill[wt] (1.5,0.866) circle (2.20pt);
\fill[wt] (0.5,0.866) circle (2.20pt);
\fill[wt] (-0.5,0.866) circle (2.20pt);
\fill[wt] (1,0) circle (2.20pt);
\fill[wt] (-1.5,0.866) circle (2.20pt);
\fill[wt] (0,0) circle (2.20pt);
\draw[wt,thin] (0,0) circle (4.60pt);
\fill[wt] (1.5,-0.866) circle (2.20pt);
\fill[wt] (-1,0) circle (2.20pt);
\fill[wt] (0.5,-0.866) circle (2.20pt);
\fill[wt] (-0.5,-0.866) circle (2.20pt);
\fill[wt] (-1.5,-0.866) circle (2.20pt);
\fill[wt] (0,-1.732) circle (2.20pt);
\draw[seg,red] (0,0) -- (1.5,-0.866);
\draw[seg,red] (-1.5,-0.866) -- (0,-1.732);
\draw[seg,red] (-1.5,0.866) -- (0,0);
\draw[seg,red] (1.5,-0.866) -- (0.5,-0.866);
\draw[seg,red] (0.5,-0.866) -- (-0.5,-0.866);
\draw[seg,red] (1.5,0.866) -- (0.5,0.866);
\draw[seg,red] (-0.5,-0.866) -- (-1.5,-0.866);
\draw[seg,red] (0.5,0.866) -- (-0.5,0.866);
\draw[seg,red] (-0.5,0.866) -- (-1.5,0.866);
\end{tikzpicture}
\\[4pt]{\small $p_{\pi}=e_{2}^{3}e_{1}^{2}e_{2}^{3}e_{1}$}
\end{minipage}
\end{center}
\subsubsection{$V_{\omega_2^\vee}$, maps $V((-1,-2))\to V((1,1))$: 6 potential $\i$-trails, 1 path}
\begin{center}
\begin{minipage}[t]{0.32\textwidth}\centering
\begin{tikzpicture}[scale=1.7]\label{fig: 7dim it paths 1}
\fill[wt] (0.5,0.866) circle (1.30pt);
\fill[wt] (-0.5,0.866) circle (1.30pt);
\fill[wt] (1,0) circle (1.30pt);
\fill[wt] (0,0) circle (1.30pt);
\fill[wt] (-1,0) circle (1.30pt);
\fill[wt] (0.5,-0.866) circle (1.30pt);
\fill[wt] (-0.5,-0.866) circle (1.30pt);
\draw[seg,red] (-1,0) -- (0.5,-0.866);
\draw[seg,red] (-0.5,0.866) -- (1,0);
\draw[seg,red] (0.5,-0.866) -- (-0.5,-0.866);
\draw[seg,red] (1,0) -- (0,0);
\draw[seg,red] (0,0) -- (-1,0);
\end{tikzpicture}
\\[4pt]{\small $p_{\pi}=e_{1}e_{2}^{2}e_{1}e_{2}$}
\end{minipage}
\end{center}
\clearpage
\subsection{\texorpdfstring{Condition (4) of \cref{defn: c-vec polytope}: trails from $\omega_i^\vee$ to $w_0s_i\omega_i^\vee$}{Condition (4)}}
\subsubsection{$V_{\omega_1^\vee}$, maps $V((-1,-3))\to V((2,3))$: 13 potential $\i$-trails, 4 paths}
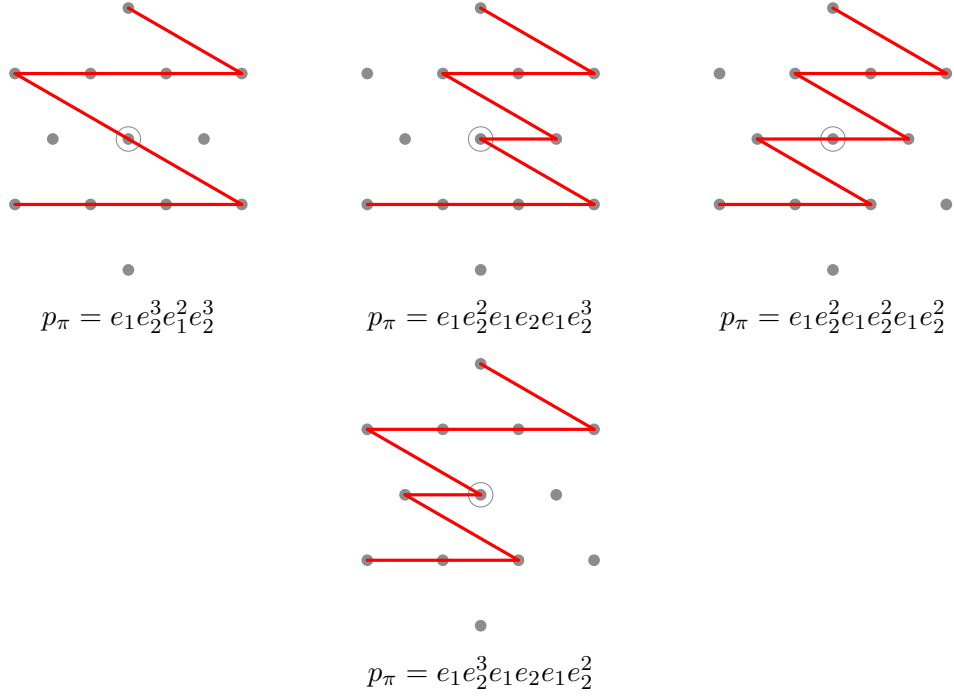
\begin{center}
\begin{minipage}[t]{0.32\textwidth}\centering
\begin{tikzpicture}[scale=1.0]
\fill[wt] (0,1.732) circle (2.20pt);
\fill[wt] (1.5,0.866) circle (2.20pt);
\fill[wt] (0.5,0.866) circle (2.20pt);
\fill[wt] (-0.5,0.866) circle (2.20pt);
\fill[wt] (1,0) circle (2.20pt);
\fill[wt] (-1.5,0.866) circle (2.20pt);
\fill[wt] (0,0) circle (2.20pt);
\draw[wt,thin] (0,0) circle (4.60pt);
\fill[wt] (1.5,-0.866) circle (2.20pt);
\fill[wt] (-1,0) circle (2.20pt);
\fill[wt] (0.5,-0.866) circle (2.20pt);
\fill[wt] (-0.5,-0.866) circle (2.20pt);
\fill[wt] (-1.5,-0.866) circle (2.20pt);
\fill[wt] (0,-1.732) circle (2.20pt);
\draw[seg,red] (0,0) -- (1.5,-0.866);
\draw[seg,red] (0,1.732) -- (1.5,0.866);
\draw[seg,red] (-1.5,0.866) -- (0,0);
\draw[seg,red] (1.5,-0.866) -- (0.5,-0.866);
\draw[seg,red] (0.5,-0.866) -- (-0.5,-0.866);
\draw[seg,red] (1.5,0.866) -- (0.5,0.866);
\draw[seg,red] (-0.5,-0.866) -- (-1.5,-0.866);
\draw[seg,red] (0.5,0.866) -- (-0.5,0.866);
\draw[seg,red] (-0.5,0.866) -- (-1.5,0.866);
\end{tikzpicture}
\\[4pt]{\small $p_{\pi}=e_{1}e_{2}^{3}e_{1}^{2}e_{2}^{3}$}
\end{minipage}\hfill
\begin{minipage}[t]{0.32\textwidth}\centering
\begin{tikzpicture}[scale=1.0]\label{fig: 14dim it paths 2}
\fill[wt] (0,1.732) circle (2.20pt);
\fill[wt] (1.5,0.866) circle (2.20pt);
\fill[wt] (0.5,0.866) circle (2.20pt);
\fill[wt] (-0.5,0.866) circle (2.20pt);
\fill[wt] (1,0) circle (2.20pt);
\fill[wt] (-1.5,0.866) circle (2.20pt);
\fill[wt] (0,0) circle (2.20pt);
\draw[wt,thin] (0,0) circle (4.60pt);
\fill[wt] (1.5,-0.866) circle (2.20pt);
\fill[wt] (-1,0) circle (2.20pt);
\fill[wt] (0.5,-0.866) circle (2.20pt);
\fill[wt] (-0.5,-0.866) circle (2.20pt);
\fill[wt] (-1.5,-0.866) circle (2.20pt);
\fill[wt] (0,-1.732) circle (2.20pt);
\draw[seg,red] (0,0) -- (1.5,-0.866);
\draw[seg,red] (-0.5,0.866) -- (1,0);
\draw[seg,red] (0,1.732) -- (1.5,0.866);
\draw[seg,red] (1.5,-0.866) -- (0.5,-0.866);
\draw[seg,red] (0.5,-0.866) -- (-0.5,-0.866);
\draw[seg,red] (1,0) -- (0,0);
\draw[seg,red] (1.5,0.866) -- (0.5,0.866);
\draw[seg,red] (-0.5,-0.866) -- (-1.5,-0.866);
\draw[seg,red] (0.5,0.866) -- (-0.5,0.866);
\end{tikzpicture}
\\[4pt]{\small $p_{\pi}=e_{1}e_{2}^{2}e_{1}e_{2}e_{1}e_{2}^{3}$}
\end{minipage}\hfill
\begin{minipage}[t]{0.32\textwidth}\centering
\begin{tikzpicture}[scale=1.0]
\fill[wt] (0,1.732) circle (2.20pt);
\fill[wt] (1.5,0.866) circle (2.20pt);
\fill[wt] (0.5,0.866) circle (2.20pt);
\fill[wt] (-0.5,0.866) circle (2.20pt);
\fill[wt] (1,0) circle (2.20pt);
\fill[wt] (-1.5,0.866) circle (2.20pt);
\fill[wt] (0,0) circle (2.20pt);
\draw[wt,thin] (0,0) circle (4.60pt);
\fill[wt] (1.5,-0.866) circle (2.20pt);
\fill[wt] (-1,0) circle (2.20pt);
\fill[wt] (0.5,-0.866) circle (2.20pt);
\fill[wt] (-0.5,-0.866) circle (2.20pt);
\fill[wt] (-1.5,-0.866) circle (2.20pt);
\fill[wt] (0,-1.732) circle (2.20pt);
\draw[seg,red] (-1,0) -- (0.5,-0.866);
\draw[seg,red] (-0.5,0.866) -- (1,0);
\draw[seg,red] (0,1.732) -- (1.5,0.866);
\draw[seg,red] (0.5,-0.866) -- (-0.5,-0.866);
\draw[seg,red] (1,0) -- (0,0);
\draw[seg,red] (1.5,0.866) -- (0.5,0.866);
\draw[seg,red] (-0.5,-0.866) -- (-1.5,-0.866);
\draw[seg,red] (0,0) -- (-1,0);
\draw[seg,red] (0.5,0.866) -- (-0.5,0.866);
\end{tikzpicture}
\\[4pt]{\small $p_{\pi}=e_{1}e_{2}^{2}e_{1}e_{2}^{2}e_{1}e_{2}^{2}$}
\end{minipage}

\bigskip

\begin{minipage}[t]{0.32\textwidth}\centering
\begin{tikzpicture}[scale=1.0]
\fill[wt] (0,1.732) circle (2.20pt);
\fill[wt] (1.5,0.866) circle (2.20pt);
\fill[wt] (0.5,0.866) circle (2.20pt);
\fill[wt] (-0.5,0.866) circle (2.20pt);
\fill[wt] (1,0) circle (2.20pt);
\fill[wt] (-1.5,0.866) circle (2.20pt);
\fill[wt] (0,0) circle (2.20pt);
\draw[wt,thin] (0,0) circle (4.60pt);
\fill[wt] (1.5,-0.866) circle (2.20pt);
\fill[wt] (-1,0) circle (2.20pt);
\fill[wt] (0.5,-0.866) circle (2.20pt);
\fill[wt] (-0.5,-0.866) circle (2.20pt);
\fill[wt] (-1.5,-0.866) circle (2.20pt);
\fill[wt] (0,-1.732) circle (2.20pt);
\draw[seg,red] (-1,0) -- (0.5,-0.866);
\draw[seg,red] (0,1.732) -- (1.5,0.866);
\draw[seg,red] (-1.5,0.866) -- (0,0);
\draw[seg,red] (0.5,-0.866) -- (-0.5,-0.866);
\draw[seg,red] (1.5,0.866) -- (0.5,0.866);
\draw[seg,red] (-0.5,-0.866) -- (-1.5,-0.866);
\draw[seg,red] (0,0) -- (-1,0);
\draw[seg,red] (0.5,0.866) -- (-0.5,0.866);
\draw[seg,red] (-0.5,0.866) -- (-1.5,0.866);
\end{tikzpicture}
\\[4pt]{\small $p_{\pi}=e_{1}e_{2}^{3}e_{1}e_{2}e_{1}e_{2}^{2}$}
\end{minipage}
\end{center}
\subsubsection{$V_{\omega_2^\vee}$, maps $V((-1,-1))\to V((1,2))$: 1 potential $\i$-trail, 1 path}
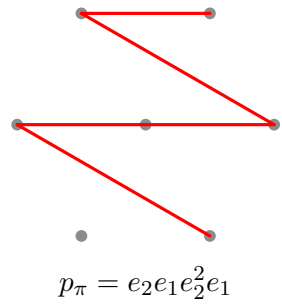
\begin{center}
\begin{minipage}[t]{0.32\textwidth}\centering
\begin{tikzpicture}[scale=1.7]\label{fig: 7dim it paths 2}
\fill[wt] (0.5,0.866) circle (1.30pt);
\fill[wt] (-0.5,0.866) circle (1.30pt);
\fill[wt] (1,0) circle (1.30pt);
\fill[wt] (0,0) circle (1.30pt);
\fill[wt] (-1,0) circle (1.30pt);
\fill[wt] (0.5,-0.866) circle (1.30pt);
\fill[wt] (-0.5,-0.866) circle (1.30pt);
\draw[seg,red] (-1,0) -- (0.5,-0.866);
\draw[seg,red] (-0.5,0.866) -- (1,0);
\draw[seg,red] (1,0) -- (0,0);
\draw[seg,red] (0,0) -- (-1,0);
\draw[seg,red] (0.5,0.866) -- (-0.5,0.866);
\end{tikzpicture}
\\[4pt]{\small $p_{\pi}=e_{2}e_{1}e_{2}^{2}e_{1}$}
\end{minipage}
\end{center}

\subsection{Triviality of the composition rule in type \texorpdfstring{$\G$}{G\unisubtwo}}
\label{subsec: Triviality of the composition rule in type G2}

So far we have found all the potential $\i$-trails necessary to compute the \(\bzlmnc\) polytopes for type \(\G\). Now all that remains is to check which potential $\i$-trails satisfy the composition rule. However, we find that in fact each of these potential $\i$-trails is an $\i$-trail.

\begin{theorem}\label{composition quasiminiscule case}
    Let \(\i\in R(w_0)\).  For $i = 1, 2$, every potential $\i$-trail from $s_{i}\omega_{i}^{\vee}$ to $w_{0}\omega_{i}^{\vee}$ and from $\omega_{i}^{\vee}$ to $w_{0}s_{i}\omega_{i}^{\vee}$ in $V_{\omega_{i}^{\vee}}$ is an $\i$-trail.
\end{theorem}

\begin{proof}
    We will prove the theorem in the case where $\i = (1,2,1,2,1,2)$. The case in which $\i = (2,1,2,1,2,1)$ is similar.
    
    We begin with the $i = 2$ case. Recall that $\omega_{2}^{\vee}$ is the short fundamental weight for $\ld{\g}$. By inspection of \cref{fig: 7dim it paths 1,fig: 7dim it paths 2}, all potential $\i$-trails from $s_{2}\omega_{2}^{\vee}$ to $w_{0}\omega_{2}^{\vee}$ have the same path and hence the same path operator. Thus, it suffices to show that one of these potential $\i$-trails is an $\i$-trail, since then this common path~operator is nonzero, from which it will follow that all of these potential $\i$-trails are $\i$-trails.
    
    To show that there is an $\i$-trail (which must be among the potential $\i$-trails) from $s_{2}\omega_{2}^{\vee}$ to $w_{0}\omega_{2}^{\vee}$ in \(V_{\co\omega_2}\), we apply \cite[Proposition 9.2]{BerZel2001}, which states that, for all \(u,v\in W\), there is an $\i$-trail from \(u\omega_{2}^{\vee}\) to \(v\omega_{2}^{\vee}\) if \(\lng{uv^{-1} } = \lng{v} - \lng{u}\) and \(uv^{-1} \leq w_0\) in the strong Bruhat order. Indeed, this is true;  furthermore, \(uv^{-1} \leq w_0\) trivially because $w_{0}$ is the largest element in Bruhat order.

    A similar argument shows that every potential $\i$-trail from \(\omega_{2}^{\vee}\) to \(w_{0}s_{2}\omega_{2}^{\vee}\) in $V_{\omega_{2}^{\vee}}$ is an $\i$-trail.

    We now consider the \(i=1\) case.  Let \(\pi\) be a potential $\i$-trail from \(s_1\co\omega_1\) to \(w_0\co\omega_1\) in \(V_{\co\omega_1}\). We see in \cref{fig: 14dim it paths 1} that there is again exactly one path that these potential $\i$-trails follow (in fact there is only one such potential $\i$-trail).  As this potential $\i$-trail path is extremal, a similar argument as in the previous paragraph holds.

    Finally we now let \(\pi\) be a potential $\i$-trail from \(\omega_{1}^{\vee}\) to \(w_{0}s_{1}\omega_{1}^{\vee}\) in \(V_{\co\omega_1}\).  Each pair $\zeta, \eta$ of consecutive weights in the potential $\i$-trail path of $\pi$ falls into one of the following three cases:
        (1)~all weight spaces in the convex hull of $\setof{\zeta, \eta}$ are one-dimensional;
        (2)~both $\zeta$ and $\eta$ have multiplicity one, but exactly one weight between them in the representation has multiplicity~$2$; or
        (3)~exactly one of $\zeta$ or $\eta$ has multiplicity $2$.
    In the first case we find the composition of the Chevalley generators is a nonzero map between one-dimensional subspaces by the root string property \cite{Hum1972}.  For the second case we need to consider how the composition \(e_je_i\) acts on \(f_i\), as this is the only composition of Chevalley generators that can show up in this manner, since our first weight space is spanned by one of the negative simple roots:
    \begin{align*}
        e_je_i\cdot f_i 
            &= 
            e_j\cdot \ad(e_i)(f_i) 
            = 
            e_j\cdot \liebracket{e_i}{f_i} 
            = 
            e_j\cdot \alpha^\vee_i \\
            &=
            \liebracket{e_j}{\alpha^\vee_i} 
            = 
            -\liebracket{\alpha^\vee_i}{e_j} 
            = 
            - \alpha_j(\alpha^\vee_i)e_j 
            = 
            - a_{i,j}e_j,
    \end{align*}
    where \(a_{i,j}\) is the \((i,j)\)-th entry in the Cartan matrix, which is never~$0$ in type~$\G$.
    In the third case, if the multiplicities of~$\zeta$ and~$\eta$ are~$1$ and~$2$ respectively then the next weight in the $\i$-trail path is of multiplicity~$1$, and again the above calculation covers this case. A similar argument holds if the weights are of multiplicity~$2$ and~$1$ respectively. 
\end{proof}

\subsection{Explicit linear equations and inequalities for \\ \texorpdfstring{$\bzlmnc$}{c-vector BZ polytopes} in type-\texorpdfstring{$\G$}{G\unisubtwo}} 
\label{subsec: Explicit linear equations and inequalities for type-G2 BZ polytopes}

As we have seen, each relevant potential $\i$-trail is an actual $\i$-trail, so we get an inequality for the BZ polytope from each $\i$-trail. In the following diagrams we show which inequality comes from each $\i$-trail. 

\subsubsection{Condition (3) of \cref{defn: c-vec polytope}: $\mathbf i$-trails from $s_i\omega_i^\vee$ to $w_0\omega_i^\vee$}
Each $\mathbf i$-trail $\pi$ in $V_{\omega_i^\vee}$ contributes $\sum_k c_k(\pi)\,t_k \ge (s_i\lambda+\nu-\mu)(\omega_i^\vee) = (\lambda+\nu-\mu)(\omega_i^\vee)-\lambda(\alpha_i^\vee)$.
\subsubsection{$i=1$: $V_{\omega_1^\vee}$, from $\gamma=(1,3)$ to $\delta=(-2,-3)$ --- 1 $\i$-trail}
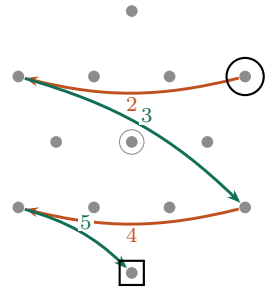
\begin{center}
\begin{minipage}[t]{0.48\textwidth}\centering
\begin{tikzpicture}[scale=1.0]
\fill[wt] (0,1.732) circle (2.2pt);
\fill[wt] (1.5,0.866) circle (2.2pt);
\fill[wt] (0.5,0.866) circle (2.2pt);
\fill[wt] (-0.5,0.866) circle (2.2pt);
\fill[wt] (1,0) circle (2.2pt);
\fill[wt] (-1.5,0.866) circle (2.2pt);
\fill[wt] (0,0) circle (2.2pt);
\draw[wt,thin] (0,0) circle (4.6pt);
\fill[wt] (1.5,-0.866) circle (2.2pt);
\fill[wt] (-1,0) circle (2.2pt);
\fill[wt] (0.5,-0.866) circle (2.2pt);
\fill[wt] (-0.5,-0.866) circle (2.2pt);
\fill[wt] (-1.5,-0.866) circle (2.2pt);
\fill[wt] (0,-1.732) circle (2.2pt);
\draw[trailstep,itwo] (1.5,0.866) to[bend left=14] (-1.5,0.866);
\node[lab,text=itwo] at (0,0.496) {2};
\draw[trailstep,ione] (-1.5,0.866) to[bend left=14] (1.5,-0.866);
\node[lab,text=ione] at (0.201,0.349) {3};
\draw[trailstep,itwo] (1.5,-0.866) to[bend left=14] (-1.5,-0.866);
\node[lab,text=itwo] at (0,-1.236) {4};
\draw[trailstep,ione] (-1.5,-0.866) to[bend left=14] (0,-1.732);
\node[lab,text=ione] at (-0.609,-1.055) {5};
\draw[thick] (1.5,0.866) circle (7pt);
\draw[thick] (-0.16,-1.892) rectangle (0.16,-1.572);
\end{tikzpicture}
\\[2pt]{\small $\pi_{1}$: $c=(0,3,2,3,1,0)$}\\{\small $3t_2+2t_3+3t_4+t_5$}\\{\small ${}\ge (s_1\lambda+\nu-\mu)(\omega_1^\vee)$}
\end{minipage}
\end{center}
\subsubsection{$i=2$: $V_{\omega_2^\vee}$, from $\gamma=(1,1)$ to $\delta=(-1,-2)$ --- 6 $\i$-trails}
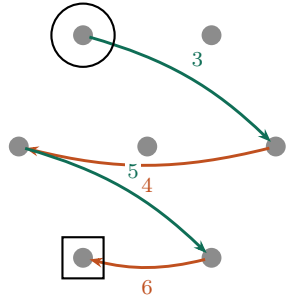
\begin{center}
\begin{minipage}[t]{0.48\textwidth}\centering
\begin{tikzpicture}[scale=1.7]
\fill[wt] (0.5,0.866) circle (2.2pt);
\fill[wt] (-0.5,0.866) circle (2.2pt);
\fill[wt] (1,0) circle (2.2pt);
\fill[wt] (0,0) circle (2.2pt);
\fill[wt] (-1,0) circle (2.2pt);
\fill[wt] (0.5,-0.866) circle (2.2pt);
\fill[wt] (-0.5,-0.866) circle (2.2pt);
\draw[trailstep,ione] (-0.5,0.866) to[bend left=14] (1,0);
\node[lab,text=ione] at (0.391,0.677) {3};
\draw[trailstep,itwo] (1,0) to[bend left=14] (-1,0);
\node[lab,text=itwo] at (0,-0.3) {4};
\draw[trailstep,ione] (-1,0) to[bend left=14] (0.5,-0.866);
\node[lab,text=ione] at (-0.109,-0.189) {5};
\draw[trailstep,itwo] (0.5,-0.866) to[bend left=14] (-0.5,-0.866);
\node[lab,text=itwo] at (0,-1.096) {6};
\draw[thick] (-0.5,0.866) circle (7pt);
\draw[thick] (-0.66,-1.026) rectangle (-0.34,-0.706);
\end{tikzpicture}
\\[2pt]{\small $\pi_{1}$: $c=(0,0,1,2,1,1)$}\\{\small $t_3+2t_4+t_5+t_6$}\\{\small ${}\ge (s_2\lambda+\nu-\mu)(\omega_2^\vee)$}
\end{minipage}\hfill
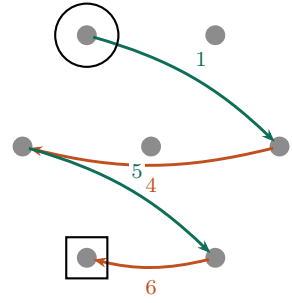
\begin{minipage}[t]{0.48\textwidth}\centering
\begin{tikzpicture}[scale=1.7]
\fill[wt] (0.5,0.866) circle (2.2pt);
\fill[wt] (-0.5,0.866) circle (2.2pt);
\fill[wt] (1,0) circle (2.2pt);
\fill[wt] (0,0) circle (2.2pt);
\fill[wt] (-1,0) circle (2.2pt);
\fill[wt] (0.5,-0.866) circle (2.2pt);
\fill[wt] (-0.5,-0.866) circle (2.2pt);
\draw[trailstep,ione] (-0.5,0.866) to[bend left=14] (1,0);
\node[lab,text=ione] at (0.391,0.677) {1};
\draw[trailstep,itwo] (1,0) to[bend left=14] (-1,0);
\node[lab,text=itwo] at (0,-0.3) {4};
\draw[trailstep,ione] (-1,0) to[bend left=14] (0.5,-0.866);
\node[lab,text=ione] at (-0.109,-0.189) {5};
\draw[trailstep,itwo] (0.5,-0.866) to[bend left=14] (-0.5,-0.866);
\node[lab,text=itwo] at (0,-1.096) {6};
\draw[thick] (-0.5,0.866) circle (7pt);
\draw[thick] (-0.66,-1.026) rectangle (-0.34,-0.706);
\end{tikzpicture}
\\[2pt]{\small $\pi_{2}$: $c=(1,0,0,2,1,1)$}\\{\small $t_1+2t_4+t_5+t_6$}\\{\small ${}\ge (s_2\lambda+\nu-\mu)(\omega_2^\vee)$}
\end{minipage}

\bigskip

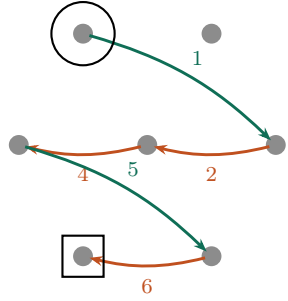
\begin{minipage}[t]{0.48\textwidth}\centering
\begin{tikzpicture}[scale=1.7]
\fill[wt] (0.5,0.866) circle (2.2pt);
\fill[wt] (-0.5,0.866) circle (2.2pt);
\fill[wt] (1,0) circle (2.2pt);
\fill[wt] (0,0) circle (2.2pt);
\fill[wt] (-1,0) circle (2.2pt);
\fill[wt] (0.5,-0.866) circle (2.2pt);
\fill[wt] (-0.5,-0.866) circle (2.2pt);
\draw[trailstep,ione] (-0.5,0.866) to[bend left=14] (1,0);
\node[lab,text=ione] at (0.391,0.677) {1};
\draw[trailstep,itwo] (1,0) to[bend left=14] (0,0);
\node[lab,text=itwo] at (0.5,-0.23) {2};
\draw[trailstep,itwo] (0,0) to[bend left=14] (-1,0);
\node[lab,text=itwo] at (-0.5,-0.23) {4};
\draw[trailstep,ione] (-1,0) to[bend left=14] (0.5,-0.866);
\node[lab,text=ione] at (-0.109,-0.189) {5};
\draw[trailstep,itwo] (0.5,-0.866) to[bend left=14] (-0.5,-0.866);
\node[lab,text=itwo] at (0,-1.096) {6};
\draw[thick] (-0.5,0.866) circle (7pt);
\draw[thick] (-0.66,-1.026) rectangle (-0.34,-0.706);
\end{tikzpicture}
\\[2pt]{\small $\pi_{3}$: $c=(1,1,0,1,1,1)$}\\{\small $t_1+t_2+t_4+t_5+t_6$}\\{\small ${}\ge (s_2\lambda+\nu-\mu)(\omega_2^\vee)$}
\end{minipage}\hfill
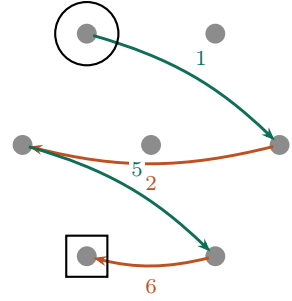
\begin{minipage}[t]{0.48\textwidth}\centering
\begin{tikzpicture}[scale=1.7]
\fill[wt] (0.5,0.866) circle (2.2pt);
\fill[wt] (-0.5,0.866) circle (2.2pt);
\fill[wt] (1,0) circle (2.2pt);
\fill[wt] (0,0) circle (2.2pt);
\fill[wt] (-1,0) circle (2.2pt);
\fill[wt] (0.5,-0.866) circle (2.2pt);
\fill[wt] (-0.5,-0.866) circle (2.2pt);
\draw[trailstep,ione] (-0.5,0.866) to[bend left=14] (1,0);
\node[lab,text=ione] at (0.391,0.677) {1};
\draw[trailstep,itwo] (1,0) to[bend left=14] (-1,0);
\node[lab,text=itwo] at (0,-0.3) {2};
\draw[trailstep,ione] (-1,0) to[bend left=14] (0.5,-0.866);
\node[lab,text=ione] at (-0.109,-0.189) {5};
\draw[trailstep,itwo] (0.5,-0.866) to[bend left=14] (-0.5,-0.866);
\node[lab,text=itwo] at (0,-1.096) {6};
\draw[thick] (-0.5,0.866) circle (7pt);
\draw[thick] (-0.66,-1.026) rectangle (-0.34,-0.706);
\end{tikzpicture}
\\[2pt]{\small $\pi_{4}$: $c=(1,2,0,0,1,1)$}\\{\small $t_1+2t_2+t_5+t_6$}\\{\small ${}\ge (s_2\lambda+\nu-\mu)(\omega_2^\vee)$}
\end{minipage}

\bigskip

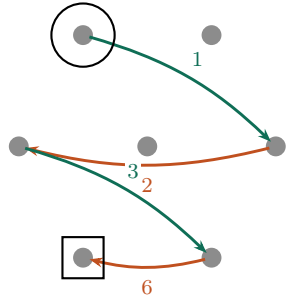
\begin{minipage}[t]{0.48\textwidth}\centering
\begin{tikzpicture}[scale=1.7]
\fill[wt] (0.5,0.866) circle (2.2pt);
\fill[wt] (-0.5,0.866) circle (2.2pt);
\fill[wt] (1,0) circle (2.2pt);
\fill[wt] (0,0) circle (2.2pt);
\fill[wt] (-1,0) circle (2.2pt);
\fill[wt] (0.5,-0.866) circle (2.2pt);
\fill[wt] (-0.5,-0.866) circle (2.2pt);
\draw[trailstep,ione] (-0.5,0.866) to[bend left=14] (1,0);
\node[lab,text=ione] at (0.391,0.677) {1};
\draw[trailstep,itwo] (1,0) to[bend left=14] (-1,0);
\node[lab,text=itwo] at (0,-0.3) {2};
\draw[trailstep,ione] (-1,0) to[bend left=14] (0.5,-0.866);
\node[lab,text=ione] at (-0.109,-0.189) {3};
\draw[trailstep,itwo] (0.5,-0.866) to[bend left=14] (-0.5,-0.866);
\node[lab,text=itwo] at (0,-1.096) {6};
\draw[thick] (-0.5,0.866) circle (7pt);
\draw[thick] (-0.66,-1.026) rectangle (-0.34,-0.706);
\end{tikzpicture}
\\[2pt]{\small $\pi_{5}$: $c=(1,2,1,0,0,1)$}\\{\small $t_1+2t_2+t_3+t_6$}\\{\small ${}\ge (s_2\lambda+\nu-\mu)(\omega_2^\vee)$}
\end{minipage}\hfill
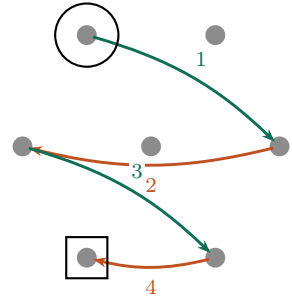
\begin{minipage}[t]{0.48\textwidth}\centering
\begin{tikzpicture}[scale=1.7]
\fill[wt] (0.5,0.866) circle (2.2pt);
\fill[wt] (-0.5,0.866) circle (2.2pt);
\fill[wt] (1,0) circle (2.2pt);
\fill[wt] (0,0) circle (2.2pt);
\fill[wt] (-1,0) circle (2.2pt);
\fill[wt] (0.5,-0.866) circle (2.2pt);
\fill[wt] (-0.5,-0.866) circle (2.2pt);
\draw[trailstep,ione] (-0.5,0.866) to[bend left=14] (1,0);
\node[lab,text=ione] at (0.391,0.677) {1};
\draw[trailstep,itwo] (1,0) to[bend left=14] (-1,0);
\node[lab,text=itwo] at (0,-0.3) {2};
\draw[trailstep,ione] (-1,0) to[bend left=14] (0.5,-0.866);
\node[lab,text=ione] at (-0.109,-0.189) {3};
\draw[trailstep,itwo] (0.5,-0.866) to[bend left=14] (-0.5,-0.866);
\node[lab,text=itwo] at (0,-1.096) {4};
\draw[thick] (-0.5,0.866) circle (7pt);
\draw[thick] (-0.66,-1.026) rectangle (-0.34,-0.706);
\end{tikzpicture}
\\[2pt]{\small $\pi_{6}$: $c=(1,2,1,1,0,0)$}\\{\small $t_1+2t_2+t_3+t_4$}\\{\small ${}\ge (s_2\lambda+\nu-\mu)(\omega_2^\vee)$}
\end{minipage}

\bigskip

\end{center}
\clearpage
\subsubsection{Condition (4) of \cref{defn: c-vec polytope}: $\mathbf i$-trails from $\omega_i^\vee$ to $w_0s_i\omega_i^\vee$}

Each $\mathbf i$-trail $\pi$ in $V_{\omega_i^\vee}$ contributes $\sum_k c_k(\pi)\,t_k \ge (\lambda+s_i\nu-\mu)(\omega_i^\vee) = (\lambda+\nu-\mu)(\omega_i^\vee)-\nu(\alpha_i^\vee)$.
\subsubsection{$i=1$: $V_{\omega_1^\vee}$, from $\gamma=(2,3)$ to $\delta=(-1,-3)$ --- 13 $\i$-trails}
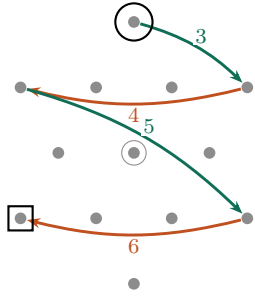
\begin{center}
\begin{minipage}[t]{0.48\textwidth}\centering
\begin{tikzpicture}[scale=1.0]
\fill[wt] (0,1.732) circle (2.2pt);
\fill[wt] (1.5,0.866) circle (2.2pt);
\fill[wt] (0.5,0.866) circle (2.2pt);
\fill[wt] (-0.5,0.866) circle (2.2pt);
\fill[wt] (1,0) circle (2.2pt);
\fill[wt] (-1.5,0.866) circle (2.2pt);
\fill[wt] (0,0) circle (2.2pt);
\draw[wt,thin] (0,0) circle (4.6pt);
\fill[wt] (1.5,-0.866) circle (2.2pt);
\fill[wt] (-1,0) circle (2.2pt);
\fill[wt] (0.5,-0.866) circle (2.2pt);
\fill[wt] (-0.5,-0.866) circle (2.2pt);
\fill[wt] (-1.5,-0.866) circle (2.2pt);
\fill[wt] (0,-1.732) circle (2.2pt);
\draw[trailstep,ione] (0,1.732) to[bend left=14] (1.5,0.866);
\node[lab,text=ione] at (0.891,1.543) {3};
\draw[trailstep,itwo] (1.5,0.866) to[bend left=14] (-1.5,0.866);
\node[lab,text=itwo] at (0,0.496) {4};
\draw[trailstep,ione] (-1.5,0.866) to[bend left=14] (1.5,-0.866);
\node[lab,text=ione] at (0.201,0.349) {5};
\draw[trailstep,itwo] (1.5,-0.866) to[bend left=14] (-1.5,-0.866);
\node[lab,text=itwo] at (0,-1.236) {6};
\draw[thick] (0,1.732) circle (7pt);
\draw[thick] (-1.66,-1.026) rectangle (-1.34,-0.706);
\end{tikzpicture}
\\[2pt]{\small $\pi_{1}$: $c=(0,0,1,3,2,3)$}\\{\small $t_3+3t_4+2t_5+3t_6$}\\{\small ${}\ge (\lambda+s_1\nu-\mu)(\omega_1^\vee)$}
\end{minipage}\hfill
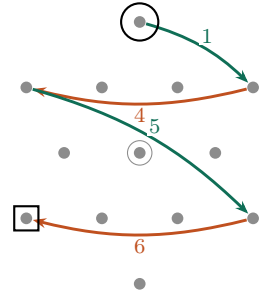
\begin{minipage}[t]{0.48\textwidth}\centering
\begin{tikzpicture}[scale=1.0]
\fill[wt] (0,1.732) circle (2.2pt);
\fill[wt] (1.5,0.866) circle (2.2pt);
\fill[wt] (0.5,0.866) circle (2.2pt);
\fill[wt] (-0.5,0.866) circle (2.2pt);
\fill[wt] (1,0) circle (2.2pt);
\fill[wt] (-1.5,0.866) circle (2.2pt);
\fill[wt] (0,0) circle (2.2pt);
\draw[wt,thin] (0,0) circle (4.6pt);
\fill[wt] (1.5,-0.866) circle (2.2pt);
\fill[wt] (-1,0) circle (2.2pt);
\fill[wt] (0.5,-0.866) circle (2.2pt);
\fill[wt] (-0.5,-0.866) circle (2.2pt);
\fill[wt] (-1.5,-0.866) circle (2.2pt);
\fill[wt] (0,-1.732) circle (2.2pt);
\draw[trailstep,ione] (0,1.732) to[bend left=14] (1.5,0.866);
\node[lab,text=ione] at (0.891,1.543) {1};
\draw[trailstep,itwo] (1.5,0.866) to[bend left=14] (-1.5,0.866);
\node[lab,text=itwo] at (0,0.496) {4};
\draw[trailstep,ione] (-1.5,0.866) to[bend left=14] (1.5,-0.866);
\node[lab,text=ione] at (0.201,0.349) {5};
\draw[trailstep,itwo] (1.5,-0.866) to[bend left=14] (-1.5,-0.866);
\node[lab,text=itwo] at (0,-1.236) {6};
\draw[thick] (0,1.732) circle (7pt);
\draw[thick] (-1.66,-1.026) rectangle (-1.34,-0.706);
\end{tikzpicture}
\\[2pt]{\small $\pi_{2}$: $c=(1,0,0,3,2,3)$}\\{\small $t_1+3t_4+2t_5+3t_6$}\\{\small ${}\ge (\lambda+s_1\nu-\mu)(\omega_1^\vee)$}
\end{minipage}

\bigskip

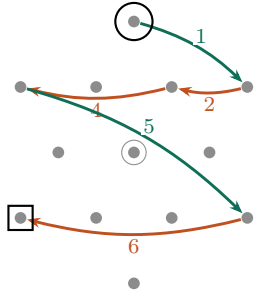
\begin{minipage}[t]{0.48\textwidth}\centering
\begin{tikzpicture}[scale=1.0]
\fill[wt] (0,1.732) circle (2.2pt);
\fill[wt] (1.5,0.866) circle (2.2pt);
\fill[wt] (0.5,0.866) circle (2.2pt);
\fill[wt] (-0.5,0.866) circle (2.2pt);
\fill[wt] (1,0) circle (2.2pt);
\fill[wt] (-1.5,0.866) circle (2.2pt);
\fill[wt] (0,0) circle (2.2pt);
\draw[wt,thin] (0,0) circle (4.6pt);
\fill[wt] (1.5,-0.866) circle (2.2pt);
\fill[wt] (-1,0) circle (2.2pt);
\fill[wt] (0.5,-0.866) circle (2.2pt);
\fill[wt] (-0.5,-0.866) circle (2.2pt);
\fill[wt] (-1.5,-0.866) circle (2.2pt);
\fill[wt] (0,-1.732) circle (2.2pt);
\draw[trailstep,ione] (0,1.732) to[bend left=14] (1.5,0.866);
\node[lab,text=ione] at (0.891,1.543) {1};
\draw[trailstep,itwo] (1.5,0.866) to[bend left=14] (0.5,0.866);
\node[lab,text=itwo] at (1,0.636) {2};
\draw[trailstep,itwo] (0.5,0.866) to[bend left=14] (-1.5,0.866);
\node[lab,text=itwo] at (-0.5,0.566) {4};
\draw[trailstep,ione] (-1.5,0.866) to[bend left=14] (1.5,-0.866);
\node[lab,text=ione] at (0.201,0.349) {5};
\draw[trailstep,itwo] (1.5,-0.866) to[bend left=14] (-1.5,-0.866);
\node[lab,text=itwo] at (0,-1.236) {6};
\draw[thick] (0,1.732) circle (7pt);
\draw[thick] (-1.66,-1.026) rectangle (-1.34,-0.706);
\end{tikzpicture}
\\[2pt]{\small $\pi_{3}$: $c=(1,1,0,2,2,3)$}\\{\small $t_1+t_2+2t_4+2t_5+3t_6$}\\{\small ${}\ge (\lambda+s_1\nu-\mu)(\omega_1^\vee)$}
\end{minipage}\hfill
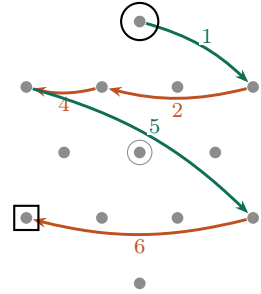
\begin{minipage}[t]{0.48\textwidth}\centering
\begin{tikzpicture}[scale=1.0]
\fill[wt] (0,1.732) circle (2.2pt);
\fill[wt] (1.5,0.866) circle (2.2pt);
\fill[wt] (0.5,0.866) circle (2.2pt);
\fill[wt] (-0.5,0.866) circle (2.2pt);
\fill[wt] (1,0) circle (2.2pt);
\fill[wt] (-1.5,0.866) circle (2.2pt);
\fill[wt] (0,0) circle (2.2pt);
\draw[wt,thin] (0,0) circle (4.6pt);
\fill[wt] (1.5,-0.866) circle (2.2pt);
\fill[wt] (-1,0) circle (2.2pt);
\fill[wt] (0.5,-0.866) circle (2.2pt);
\fill[wt] (-0.5,-0.866) circle (2.2pt);
\fill[wt] (-1.5,-0.866) circle (2.2pt);
\fill[wt] (0,-1.732) circle (2.2pt);
\draw[trailstep,ione] (0,1.732) to[bend left=14] (1.5,0.866);
\node[lab,text=ione] at (0.891,1.543) {1};
\draw[trailstep,itwo] (1.5,0.866) to[bend left=14] (-0.5,0.866);
\node[lab,text=itwo] at (0.5,0.566) {2};
\draw[trailstep,itwo] (-0.5,0.866) to[bend left=14] (-1.5,0.866);
\node[lab,text=itwo] at (-1,0.636) {4};
\draw[trailstep,ione] (-1.5,0.866) to[bend left=14] (1.5,-0.866);
\node[lab,text=ione] at (0.201,0.349) {5};
\draw[trailstep,itwo] (1.5,-0.866) to[bend left=14] (-1.5,-0.866);
\node[lab,text=itwo] at (0,-1.236) {6};
\draw[thick] (0,1.732) circle (7pt);
\draw[thick] (-1.66,-1.026) rectangle (-1.34,-0.706);
\end{tikzpicture}
\\[2pt]{\small $\pi_{4}$: $c=(1,2,0,1,2,3)$}\\{\small $t_1+2t_2+t_4+2t_5+3t_6$}\\{\small ${}\ge (\lambda+s_1\nu-\mu)(\omega_1^\vee)$}
\end{minipage}

\bigskip

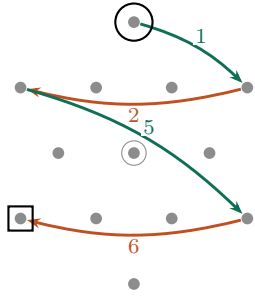
\begin{minipage}[t]{0.48\textwidth}\centering
\begin{tikzpicture}[scale=1.0]
\fill[wt] (0,1.732) circle (2.2pt);
\fill[wt] (1.5,0.866) circle (2.2pt);
\fill[wt] (0.5,0.866) circle (2.2pt);
\fill[wt] (-0.5,0.866) circle (2.2pt);
\fill[wt] (1,0) circle (2.2pt);
\fill[wt] (-1.5,0.866) circle (2.2pt);
\fill[wt] (0,0) circle (2.2pt);
\draw[wt,thin] (0,0) circle (4.6pt);
\fill[wt] (1.5,-0.866) circle (2.2pt);
\fill[wt] (-1,0) circle (2.2pt);
\fill[wt] (0.5,-0.866) circle (2.2pt);
\fill[wt] (-0.5,-0.866) circle (2.2pt);
\fill[wt] (-1.5,-0.866) circle (2.2pt);
\fill[wt] (0,-1.732) circle (2.2pt);
\draw[trailstep,ione] (0,1.732) to[bend left=14] (1.5,0.866);
\node[lab,text=ione] at (0.891,1.543) {1};
\draw[trailstep,itwo] (1.5,0.866) to[bend left=14] (-1.5,0.866);
\node[lab,text=itwo] at (0,0.496) {2};
\draw[trailstep,ione] (-1.5,0.866) to[bend left=14] (1.5,-0.866);
\node[lab,text=ione] at (0.201,0.349) {5};
\draw[trailstep,itwo] (1.5,-0.866) to[bend left=14] (-1.5,-0.866);
\node[lab,text=itwo] at (0,-1.236) {6};
\draw[thick] (0,1.732) circle (7pt);
\draw[thick] (-1.66,-1.026) rectangle (-1.34,-0.706);
\end{tikzpicture}
\\[2pt]{\small $\pi_{5}$: $c=(1,3,0,0,2,3)$}\\{\small $t_1+3t_2+2t_5+3t_6$}\\{\small ${}\ge (\lambda+s_1\nu-\mu)(\omega_1^\vee)$}
\end{minipage}\hfill
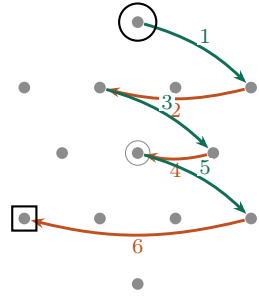
\begin{minipage}[t]{0.48\textwidth}\centering
\begin{tikzpicture}[scale=1.0]
\fill[wt] (0,1.732) circle (2.2pt);
\fill[wt] (1.5,0.866) circle (2.2pt);
\fill[wt] (0.5,0.866) circle (2.2pt);
\fill[wt] (-0.5,0.866) circle (2.2pt);
\fill[wt] (1,0) circle (2.2pt);
\fill[wt] (-1.5,0.866) circle (2.2pt);
\fill[wt] (0,0) circle (2.2pt);
\draw[wt,thin] (0,0) circle (4.6pt);
\fill[wt] (1.5,-0.866) circle (2.2pt);
\fill[wt] (-1,0) circle (2.2pt);
\fill[wt] (0.5,-0.866) circle (2.2pt);
\fill[wt] (-0.5,-0.866) circle (2.2pt);
\fill[wt] (-1.5,-0.866) circle (2.2pt);
\fill[wt] (0,-1.732) circle (2.2pt);
\draw[trailstep,ione] (0,1.732) to[bend left=14] (1.5,0.866);
\node[lab,text=ione] at (0.891,1.543) {1};
\draw[trailstep,itwo] (1.5,0.866) to[bend left=14] (-0.5,0.866);
\node[lab,text=itwo] at (0.5,0.566) {2};
\draw[trailstep,ione] (-0.5,0.866) to[bend left=14] (1,0);
\node[lab,text=ione] at (0.391,0.677) {3};
\draw[trailstep,itwo] (1,0) to[bend left=14] (0,0);
\node[lab,text=itwo] at (0.5,-0.23) {4};
\draw[trailstep,ione] (0,0) to[bend left=14] (1.5,-0.866);
\node[lab,text=ione] at (0.891,-0.189) {5};
\draw[trailstep,itwo] (1.5,-0.866) to[bend left=14] (-1.5,-0.866);
\node[lab,text=itwo] at (0,-1.236) {6};
\draw[thick] (0,1.732) circle (7pt);
\draw[thick] (-1.66,-1.026) rectangle (-1.34,-0.706);
\end{tikzpicture}
\\[2pt]{\small $\pi_{6}$: $c=(1,2,1,1,1,3)$}\\{\small $t_1+2t_2+t_3+t_4+t_5+3t_6$}\\{\small ${}\ge (\lambda+s_1\nu-\mu)(\omega_1^\vee)$}
\end{minipage}

\bigskip

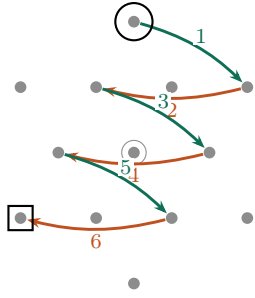
\begin{minipage}[t]{0.48\textwidth}\centering
\begin{tikzpicture}[scale=1.0]
\fill[wt] (0,1.732) circle (2.2pt);
\fill[wt] (1.5,0.866) circle (2.2pt);
\fill[wt] (0.5,0.866) circle (2.2pt);
\fill[wt] (-0.5,0.866) circle (2.2pt);
\fill[wt] (1,0) circle (2.2pt);
\fill[wt] (-1.5,0.866) circle (2.2pt);
\fill[wt] (0,0) circle (2.2pt);
\draw[wt,thin] (0,0) circle (4.6pt);
\fill[wt] (1.5,-0.866) circle (2.2pt);
\fill[wt] (-1,0) circle (2.2pt);
\fill[wt] (0.5,-0.866) circle (2.2pt);
\fill[wt] (-0.5,-0.866) circle (2.2pt);
\fill[wt] (-1.5,-0.866) circle (2.2pt);
\fill[wt] (0,-1.732) circle (2.2pt);
\draw[trailstep,ione] (0,1.732) to[bend left=14] (1.5,0.866);
\node[lab,text=ione] at (0.891,1.543) {1};
\draw[trailstep,itwo] (1.5,0.866) to[bend left=14] (-0.5,0.866);
\node[lab,text=itwo] at (0.5,0.566) {2};
\draw[trailstep,ione] (-0.5,0.866) to[bend left=14] (1,0);
\node[lab,text=ione] at (0.391,0.677) {3};
\draw[trailstep,itwo] (1,0) to[bend left=14] (-1,0);
\node[lab,text=itwo] at (0,-0.3) {4};
\draw[trailstep,ione] (-1,0) to[bend left=14] (0.5,-0.866);
\node[lab,text=ione] at (-0.109,-0.189) {5};
\draw[trailstep,itwo] (0.5,-0.866) to[bend left=14] (-1.5,-0.866);
\node[lab,text=itwo] at (-0.5,-1.166) {6};
\draw[thick] (0,1.732) circle (7pt);
\draw[thick] (-1.66,-1.026) rectangle (-1.34,-0.706);
\end{tikzpicture}
\\[2pt]{\small $\pi_{7}$: $c=(1,2,1,2,1,2)$}\\{\small $t_1+2t_2+t_3+2t_4+t_5+2t_6$}\\{\small ${}\ge (\lambda+s_1\nu-\mu)(\omega_1^\vee)$}
\end{minipage}\hfill
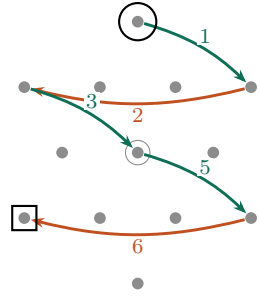
\begin{minipage}[t]{0.48\textwidth}\centering
\begin{tikzpicture}[scale=1.0]
\fill[wt] (0,1.732) circle (2.2pt);
\fill[wt] (1.5,0.866) circle (2.2pt);
\fill[wt] (0.5,0.866) circle (2.2pt);
\fill[wt] (-0.5,0.866) circle (2.2pt);
\fill[wt] (1,0) circle (2.2pt);
\fill[wt] (-1.5,0.866) circle (2.2pt);
\fill[wt] (0,0) circle (2.2pt);
\draw[wt,thin] (0,0) circle (4.6pt);
\fill[wt] (1.5,-0.866) circle (2.2pt);
\fill[wt] (-1,0) circle (2.2pt);
\fill[wt] (0.5,-0.866) circle (2.2pt);
\fill[wt] (-0.5,-0.866) circle (2.2pt);
\fill[wt] (-1.5,-0.866) circle (2.2pt);
\fill[wt] (0,-1.732) circle (2.2pt);
\draw[trailstep,ione] (0,1.732) to[bend left=14] (1.5,0.866);
\node[lab,text=ione] at (0.891,1.543) {1};
\draw[trailstep,itwo] (1.5,0.866) to[bend left=14] (-1.5,0.866);
\node[lab,text=itwo] at (0,0.496) {2};
\draw[trailstep,ione] (-1.5,0.866) to[bend left=14] (0,0);
\node[lab,text=ione] at (-0.609,0.677) {3};
\draw[trailstep,ione] (0,0) to[bend left=14] (1.5,-0.866);
\node[lab,text=ione] at (0.891,-0.189) {5};
\draw[trailstep,itwo] (1.5,-0.866) to[bend left=14] (-1.5,-0.866);
\node[lab,text=itwo] at (0,-1.236) {6};
\draw[thick] (0,1.732) circle (7pt);
\draw[thick] (-1.66,-1.026) rectangle (-1.34,-0.706);
\end{tikzpicture}
\\[2pt]{\small $\pi_{8}$: $c=(1,3,1,0,1,3)$}\\{\small $t_1+3t_2+t_3+t_5+3t_6$}\\{\small ${}\ge (\lambda+s_1\nu-\mu)(\omega_1^\vee)$}
\end{minipage}

\bigskip

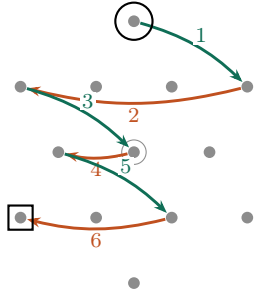
\begin{minipage}[t]{0.48\textwidth}\centering
\begin{tikzpicture}[scale=1.0]
\fill[wt] (0,1.732) circle (2.2pt);
\fill[wt] (1.5,0.866) circle (2.2pt);
\fill[wt] (0.5,0.866) circle (2.2pt);
\fill[wt] (-0.5,0.866) circle (2.2pt);
\fill[wt] (1,0) circle (2.2pt);
\fill[wt] (-1.5,0.866) circle (2.2pt);
\fill[wt] (0,0) circle (2.2pt);
\draw[wt,thin] (0,0) circle (4.6pt);
\fill[wt] (1.5,-0.866) circle (2.2pt);
\fill[wt] (-1,0) circle (2.2pt);
\fill[wt] (0.5,-0.866) circle (2.2pt);
\fill[wt] (-0.5,-0.866) circle (2.2pt);
\fill[wt] (-1.5,-0.866) circle (2.2pt);
\fill[wt] (0,-1.732) circle (2.2pt);
\draw[trailstep,ione] (0,1.732) to[bend left=14] (1.5,0.866);
\node[lab,text=ione] at (0.891,1.543) {1};
\draw[trailstep,itwo] (1.5,0.866) to[bend left=14] (-1.5,0.866);
\node[lab,text=itwo] at (0,0.496) {2};
\draw[trailstep,ione] (-1.5,0.866) to[bend left=14] (0,0);
\node[lab,text=ione] at (-0.609,0.677) {3};
\draw[trailstep,itwo] (0,0) to[bend left=14] (-1,0);
\node[lab,text=itwo] at (-0.5,-0.23) {4};
\draw[trailstep,ione] (-1,0) to[bend left=14] (0.5,-0.866);
\node[lab,text=ione] at (-0.109,-0.189) {5};
\draw[trailstep,itwo] (0.5,-0.866) to[bend left=14] (-1.5,-0.866);
\node[lab,text=itwo] at (-0.5,-1.166) {6};
\draw[thick] (0,1.732) circle (7pt);
\draw[thick] (-1.66,-1.026) rectangle (-1.34,-0.706);
\end{tikzpicture}
\\[2pt]{\small $\pi_{9}$: $c=(1,3,1,1,1,2)$}\\{\small $t_1+3t_2+t_3+t_4+t_5+2t_6$}\\{\small ${}\ge (\lambda+s_1\nu-\mu)(\omega_1^\vee)$}
\end{minipage}\hfill
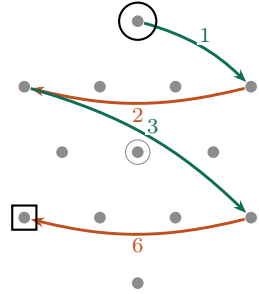
\begin{minipage}[t]{0.48\textwidth}\centering
\begin{tikzpicture}[scale=1.0]
\fill[wt] (0,1.732) circle (2.2pt);
\fill[wt] (1.5,0.866) circle (2.2pt);
\fill[wt] (0.5,0.866) circle (2.2pt);
\fill[wt] (-0.5,0.866) circle (2.2pt);
\fill[wt] (1,0) circle (2.2pt);
\fill[wt] (-1.5,0.866) circle (2.2pt);
\fill[wt] (0,0) circle (2.2pt);
\draw[wt,thin] (0,0) circle (4.6pt);
\fill[wt] (1.5,-0.866) circle (2.2pt);
\fill[wt] (-1,0) circle (2.2pt);
\fill[wt] (0.5,-0.866) circle (2.2pt);
\fill[wt] (-0.5,-0.866) circle (2.2pt);
\fill[wt] (-1.5,-0.866) circle (2.2pt);
\fill[wt] (0,-1.732) circle (2.2pt);
\draw[trailstep,ione] (0,1.732) to[bend left=14] (1.5,0.866);
\node[lab,text=ione] at (0.891,1.543) {1};
\draw[trailstep,itwo] (1.5,0.866) to[bend left=14] (-1.5,0.866);
\node[lab,text=itwo] at (0,0.496) {2};
\draw[trailstep,ione] (-1.5,0.866) to[bend left=14] (1.5,-0.866);
\node[lab,text=ione] at (0.201,0.349) {3};
\draw[trailstep,itwo] (1.5,-0.866) to[bend left=14] (-1.5,-0.866);
\node[lab,text=itwo] at (0,-1.236) {6};
\draw[thick] (0,1.732) circle (7pt);
\draw[thick] (-1.66,-1.026) rectangle (-1.34,-0.706);
\end{tikzpicture}
\\[2pt]{\small $\pi_{10}$: $c=(1,3,2,0,0,3)$}\\{\small $t_1+3t_2+2t_3+3t_6$}\\{\small ${}\ge (\lambda+s_1\nu-\mu)(\omega_1^\vee)$}
\end{minipage}

\bigskip

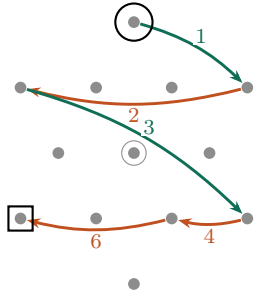
\begin{minipage}[t]{0.48\textwidth}\centering
\begin{tikzpicture}[scale=1.0]
\fill[wt] (0,1.732) circle (2.2pt);
\fill[wt] (1.5,0.866) circle (2.2pt);
\fill[wt] (0.5,0.866) circle (2.2pt);
\fill[wt] (-0.5,0.866) circle (2.2pt);
\fill[wt] (1,0) circle (2.2pt);
\fill[wt] (-1.5,0.866) circle (2.2pt);
\fill[wt] (0,0) circle (2.2pt);
\draw[wt,thin] (0,0) circle (4.6pt);
\fill[wt] (1.5,-0.866) circle (2.2pt);
\fill[wt] (-1,0) circle (2.2pt);
\fill[wt] (0.5,-0.866) circle (2.2pt);
\fill[wt] (-0.5,-0.866) circle (2.2pt);
\fill[wt] (-1.5,-0.866) circle (2.2pt);
\fill[wt] (0,-1.732) circle (2.2pt);
\draw[trailstep,ione] (0,1.732) to[bend left=14] (1.5,0.866);
\node[lab,text=ione] at (0.891,1.543) {1};
\draw[trailstep,itwo] (1.5,0.866) to[bend left=14] (-1.5,0.866);
\node[lab,text=itwo] at (0,0.496) {2};
\draw[trailstep,ione] (-1.5,0.866) to[bend left=14] (1.5,-0.866);
\node[lab,text=ione] at (0.201,0.349) {3};
\draw[trailstep,itwo] (1.5,-0.866) to[bend left=14] (0.5,-0.866);
\node[lab,text=itwo] at (1,-1.096) {4};
\draw[trailstep,itwo] (0.5,-0.866) to[bend left=14] (-1.5,-0.866);
\node[lab,text=itwo] at (-0.5,-1.166) {6};
\draw[thick] (0,1.732) circle (7pt);
\draw[thick] (-1.66,-1.026) rectangle (-1.34,-0.706);
\end{tikzpicture}
\\[2pt]{\small $\pi_{11}$: $c=(1,3,2,1,0,2)$}\\{\small $t_1+3t_2+2t_3+t_4+2t_6$}\\{\small ${}\ge (\lambda+s_1\nu-\mu)(\omega_1^\vee)$}
\end{minipage}\hfill
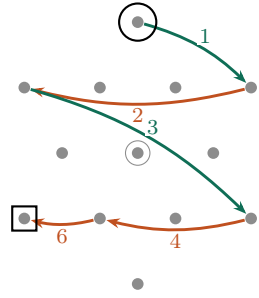
\begin{minipage}[t]{0.48\textwidth}\centering
\begin{tikzpicture}[scale=1.0]
\fill[wt] (0,1.732) circle (2.2pt);
\fill[wt] (1.5,0.866) circle (2.2pt);
\fill[wt] (0.5,0.866) circle (2.2pt);
\fill[wt] (-0.5,0.866) circle (2.2pt);
\fill[wt] (1,0) circle (2.2pt);
\fill[wt] (-1.5,0.866) circle (2.2pt);
\fill[wt] (0,0) circle (2.2pt);
\draw[wt,thin] (0,0) circle (4.6pt);
\fill[wt] (1.5,-0.866) circle (2.2pt);
\fill[wt] (-1,0) circle (2.2pt);
\fill[wt] (0.5,-0.866) circle (2.2pt);
\fill[wt] (-0.5,-0.866) circle (2.2pt);
\fill[wt] (-1.5,-0.866) circle (2.2pt);
\fill[wt] (0,-1.732) circle (2.2pt);
\draw[trailstep,ione] (0,1.732) to[bend left=14] (1.5,0.866);
\node[lab,text=ione] at (0.891,1.543) {1};
\draw[trailstep,itwo] (1.5,0.866) to[bend left=14] (-1.5,0.866);
\node[lab,text=itwo] at (0,0.496) {2};
\draw[trailstep,ione] (-1.5,0.866) to[bend left=14] (1.5,-0.866);
\node[lab,text=ione] at (0.201,0.349) {3};
\draw[trailstep,itwo] (1.5,-0.866) to[bend left=14] (-0.5,-0.866);
\node[lab,text=itwo] at (0.5,-1.166) {4};
\draw[trailstep,itwo] (-0.5,-0.866) to[bend left=14] (-1.5,-0.866);
\node[lab,text=itwo] at (-1,-1.096) {6};
\draw[thick] (0,1.732) circle (7pt);
\draw[thick] (-1.66,-1.026) rectangle (-1.34,-0.706);
\end{tikzpicture}
\\[2pt]{\small $\pi_{12}$: $c=(1,3,2,2,0,1)$}\\{\small $t_1+3t_2+2t_3+2t_4+t_6$}\\{\small ${}\ge (\lambda+s_1\nu-\mu)(\omega_1^\vee)$}
\end{minipage}

\bigskip

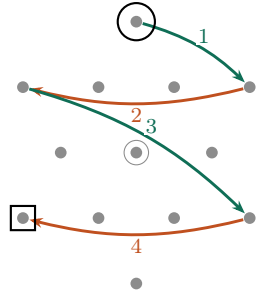
\begin{minipage}[t]{0.48\textwidth}\centering
\begin{tikzpicture}[scale=1.0]
\fill[wt] (0,1.732) circle (2.2pt);
\fill[wt] (1.5,0.866) circle (2.2pt);
\fill[wt] (0.5,0.866) circle (2.2pt);
\fill[wt] (-0.5,0.866) circle (2.2pt);
\fill[wt] (1,0) circle (2.2pt);
\fill[wt] (-1.5,0.866) circle (2.2pt);
\fill[wt] (0,0) circle (2.2pt);
\draw[wt,thin] (0,0) circle (4.6pt);
\fill[wt] (1.5,-0.866) circle (2.2pt);
\fill[wt] (-1,0) circle (2.2pt);
\fill[wt] (0.5,-0.866) circle (2.2pt);
\fill[wt] (-0.5,-0.866) circle (2.2pt);
\fill[wt] (-1.5,-0.866) circle (2.2pt);
\fill[wt] (0,-1.732) circle (2.2pt);
\draw[trailstep,ione] (0,1.732) to[bend left=14] (1.5,0.866);
\node[lab,text=ione] at (0.891,1.543) {1};
\draw[trailstep,itwo] (1.5,0.866) to[bend left=14] (-1.5,0.866);
\node[lab,text=itwo] at (0,0.496) {2};
\draw[trailstep,ione] (-1.5,0.866) to[bend left=14] (1.5,-0.866);
\node[lab,text=ione] at (0.201,0.349) {3};
\draw[trailstep,itwo] (1.5,-0.866) to[bend left=14] (-1.5,-0.866);
\node[lab,text=itwo] at (0,-1.236) {4};
\draw[thick] (0,1.732) circle (7pt);
\draw[thick] (-1.66,-1.026) rectangle (-1.34,-0.706);
\end{tikzpicture}
\\[2pt]{\small $\pi_{13}$: $c=(1,3,2,3,0,0)$}\\{\small $t_1+3t_2+2t_3+3t_4$}\\{\small ${}\ge (\lambda+s_1\nu-\mu)(\omega_1^\vee)$}
\end{minipage}
\end{center}
\subsubsection{$i=2$: $V_{\omega_2^\vee}$, from $\gamma=(1,2)$ to $\delta=(-1,-1)$ --- 1 $\i$-trail}
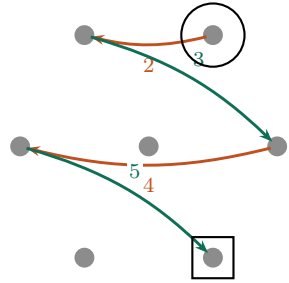
\begin{center}
\begin{minipage}[t]{0.48\textwidth}\centering
\begin{tikzpicture}[scale=1.7]
\fill[wt] (0.5,0.866) circle (2.2pt);
\fill[wt] (-0.5,0.866) circle (2.2pt);
\fill[wt] (1,0) circle (2.2pt);
\fill[wt] (0,0) circle (2.2pt);
\fill[wt] (-1,0) circle (2.2pt);
\fill[wt] (0.5,-0.866) circle (2.2pt);
\fill[wt] (-0.5,-0.866) circle (2.2pt);
\draw[trailstep,itwo] (0.5,0.866) to[bend left=14] (-0.5,0.866);
\node[lab,text=itwo] at (0,0.636) {2};
\draw[trailstep,ione] (-0.5,0.866) to[bend left=14] (1,0);
\node[lab,text=ione] at (0.391,0.677) {3};
\draw[trailstep,itwo] (1,0) to[bend left=14] (-1,0);
\node[lab,text=itwo] at (0,-0.3) {4};
\draw[trailstep,ione] (-1,0) to[bend left=14] (0.5,-0.866);
\node[lab,text=ione] at (-0.109,-0.189) {5};
\draw[thick] (0.5,0.866) circle (7pt);
\draw[thick] (0.34,-1.026) rectangle (0.66,-0.706);
\end{tikzpicture}
\\[2pt]{\small $\pi_{1}$: $c=(0,1,1,2,1,0)$}\\{\small $t_2+t_3+2t_4+t_5$}\\{\small ${}\ge (\lambda+s_2\nu-\mu)(\omega_2^\vee)$}
\end{minipage}
\end{center}
\clearpage

We summarize these results in the following table.
\begin{longtable}{lllll}
condition & $i$ & trail & $c(\pi)$ & inequality\\\hline
(3) & 1 & $\pi_{1}$ & $(0,3,2,3,1,0)$ & $3t_2+2t_3+3t_4+t_5\ge (s_1\lambda+\nu-\mu)(\omega_1^\vee)$\\
(3) & 2 & $\pi_{1}$ & $(0,0,1,2,1,1)$ & $t_3+2t_4+t_5+t_6\ge (s_2\lambda+\nu-\mu)(\omega_2^\vee)$\\
(3) & 2 & $\pi_{2}$ & $(1,0,0,2,1,1)$ & $t_1+2t_4+t_5+t_6\ge (s_2\lambda+\nu-\mu)(\omega_2^\vee)$\\
(3) & 2 & $\pi_{3}$ & $(1,1,0,1,1,1)$ & $t_1+t_2+t_4+t_5+t_6\ge (s_2\lambda+\nu-\mu)(\omega_2^\vee)$\\
(3) & 2 & $\pi_{4}$ & $(1,2,0,0,1,1)$ & $t_1+2t_2+t_5+t_6\ge (s_2\lambda+\nu-\mu)(\omega_2^\vee)$\\
(3) & 2 & $\pi_{5}$ & $(1,2,1,0,0,1)$ & $t_1+2t_2+t_3+t_6\ge (s_2\lambda+\nu-\mu)(\omega_2^\vee)$\\
(3) & 2 & $\pi_{6}$ & $(1,2,1,1,0,0)$ & $t_1+2t_2+t_3+t_4\ge (s_2\lambda+\nu-\mu)(\omega_2^\vee)$\\
(4) & 1 & $\pi_{1}$ & $(0,0,1,3,2,3)$ & $t_3+3t_4+2t_5+3t_6\ge (\lambda+s_1\nu-\mu)(\omega_1^\vee)$\\
(4) & 1 & $\pi_{2}$ & $(1,0,0,3,2,3)$ & $t_1+3t_4+2t_5+3t_6\ge (\lambda+s_1\nu-\mu)(\omega_1^\vee)$\\
(4) & 1 & $\pi_{3}$ & $(1,1,0,2,2,3)$ & $t_1+t_2+2t_4+2t_5+3t_6\ge (\lambda+s_1\nu-\mu)(\omega_1^\vee)$\\
(4) & 1 & $\pi_{4}$ & $(1,2,0,1,2,3)$ & $t_1+2t_2+t_4+2t_5+3t_6\ge (\lambda+s_1\nu-\mu)(\omega_1^\vee)$\\
(4) & 1 & $\pi_{5}$ & $(1,3,0,0,2,3)$ & $t_1+3t_2+2t_5+3t_6\ge (\lambda+s_1\nu-\mu)(\omega_1^\vee)$\\
(4) & 1 & $\pi_{6}$ & $(1,2,1,1,1,3)$ & $t_1+2t_2+t_3+t_4+t_5+3t_6\ge (\lambda+s_1\nu-\mu)(\omega_1^\vee)$\\
(4) & 1 & $\pi_{7}$ & $(1,2,1,2,1,2)$ & $t_1+2t_2+t_3+2t_4+t_5+2t_6\ge (\lambda+s_1\nu-\mu)(\omega_1^\vee)$\\
(4) & 1 & $\pi_{8}$ & $(1,3,1,0,1,3)$ & $t_1+3t_2+t_3+t_5+3t_6\ge (\lambda+s_1\nu-\mu)(\omega_1^\vee)$\\
(4) & 1 & $\pi_{9}$ & $(1,3,1,1,1,2)$ & $t_1+3t_2+t_3+t_4+t_5+2t_6\ge (\lambda+s_1\nu-\mu)(\omega_1^\vee)$\\
(4) & 1 & $\pi_{10}$ & $(1,3,2,0,0,3)$ & $t_1+3t_2+2t_3+3t_6\ge (\lambda+s_1\nu-\mu)(\omega_1^\vee)$\\
(4) & 1 & $\pi_{11}$ & $(1,3,2,1,0,2)$ & $t_1+3t_2+2t_3+t_4+2t_6\ge (\lambda+s_1\nu-\mu)(\omega_1^\vee)$\\
(4) & 1 & $\pi_{12}$ & $(1,3,2,2,0,1)$ & $t_1+3t_2+2t_3+2t_4+t_6\ge (\lambda+s_1\nu-\mu)(\omega_1^\vee)$\\
(4) & 1 & $\pi_{13}$ & $(1,3,2,3,0,0)$ & $t_1+3t_2+2t_3+3t_4\ge (\lambda+s_1\nu-\mu)(\omega_1^\vee)$\\
(4) & 2 & $\pi_{1}$ & $(0,1,1,2,1,0)$ & $t_2+t_3+2t_4+t_5\ge (\lambda+s_2\nu-\mu)(\omega_2^\vee)$\\
\end{longtable}
These inequalities, together with positivity and the equalities
\begin{equation}
    \sum_{k=1}^\ell t_ks_{i_1}\cdots s_{i_{k-1}}\alpha_{i_k}=\lambda+\nu-\mu
\end{equation}
define the BZ polytope presented in \cref{thm: explicit equations and inequalities for type G2}.

\section{Quasi-periods of BZ polytopes of type \texorpdfstring{$\G$}{G\unisubtwo}}
\label{sec: Quasi-periods of BZ polytopes of type G2}
              
We now show that the minimum quasi-period of every stretched Clebsch--Gordan coefficient for type~$\G$ divides $6$. Equivalently, we show that the minimum quasi-period of the Ehrhart quasi-polynomial of every type-$\G$ BZ~polytope divides $6$. Recall that every stretched Kostka number is a stretched Clebsch--Gordan coefficient. Thus, this result explains the quasi-period $6$ observed in \cite[Section~4]{BJK2025}.

As in \cite[Proposition~2.6]{DeLMcA2004} and \cite[Proposition 1.2]{DeLMcA2006}, our method will be to apply the theory of vector partition functions \cite{Stu1995} to the Steinberg multiplicity formula. We now recall this formula. Fix an ordering $(\beta_{1}, \dotsc, \beta_{m})$ of the set $\posroots$ of positive roots of $\g$. Define the \defing{Kostant partition function} $p \maps \weights \to \Znn$ by
\begin{equation*}
    p(\kappa) 
        \deftobe
        \#
        \setof{
            x \in \Znn^{m}
            \st
            \posroots x = \kappa
        }, 
\end{equation*}
where we write $\posroots x \deftobe \xi_{1}\beta_{1} + \dotsb + \xi_{m} \beta_{m}$ for all $x = (\xi_{1}, \dotsc, \xi_{m}) \in \Z^{m}$.  Let $\rho \deftobe \frac{1}{2}\sum_{i = 1}^{m} \beta_{i}$. The \defing{Steinberg multiplicity formula} \cite[Section 24.4]{Hum1972} states that
\begin{align*}
    \clnm
        &=
        \sum_{(\wa, \wb) \in \Wg \times \Wg}
        (-1)^{\wa\wb}
        p\bigparens{
            (
                \wa(\lambda) 
                +
                \wb(\nu)
                -
                \mu
            )
            +
            (
                \wa(\rho)
                +
                \wb(\rho) 
                - 
                2\rho
            )
        }
\end{align*}
for all dominant weights $\lambda, \mu, \nu \in \weights \cap \Wc$. We now use this result to prove \cref{thm: quasi-period at most 6}, which we repeat here for the reader's convenience.

\begin{theorem}
    Let $\g$ be a type-$\G$ Lie algebra, and let $\lambda$, $\mu$ and $\nu$ be dominant weights of~$\g$. Then the minimum quasi-period of the stretched Clebsch--Gordan coefficient $t \mapsto c_{t\lambda, t\nu}^{t\mu}$ divides $6$.
\end{theorem}

\begin{proof}
    Writing $\posrootsmatrix$ for the matrix whose columns are the positive roots of~$\g$ written in the standard basis \cite[p.~65]{Hum1972}, we have that
    \begin{equation*}
        \posrootsmatrix
            = 
            \begin{bmatrix*}[r]
                 1 &  0 &  1 &  2 & -1 & -1 \\
                -1 &  1 &  0 & -1 &  2 & -1 \\
                 0 & -1 & -1 & -1 & -1 &  2
            \end{bmatrix*}.
    \end{equation*}
    By Steinberg's multiplicity formula, the stretched Clebsch--Gordan coefficient $t \mapsto c_{t\lambda, t\nu}^{t\mu}$ is an alternating sum of functions of the form
    \begin{equation}\label{function: Steinberg term}
        t
            \mapsto
            \#
            \setof{
                x \in \Znn^{6}
                \st
                \posrootsmatrix x = [\kappa_{t}]
            },
    \end{equation}
    where $\kappa_{t} \deftobe t\bigparens{\wa(\lambda) + \wb(\nu) - \mu} + \bigparens{\wa(\rho) + \wb(\rho) - 2\rho}$ for fixed $(\wa, \wb) \in \Wg \times \Wg$ and $\lambda, \mu, \nu \in \Wc$, and $[\kappa_{t}]$ denotes the column vector of $\kappa_{t}$ in the standard basis. Note that some fixed chamber of the vector partition function for $\posrootsmatrix$ contains $[\kappa_{t}]$ for all $t \gg 0$, provided that the function on the right-hand side of~\eqref{function: Steinberg term} ever takes a nonzero value. Hence, the function \labelcref{function: Steinberg term} is a fixed quasipolynomial function of~$t$ for all $t \gg 0$, with a minimum quasi-period dividing a quasi-period of the vector partition function of $\posrootsmatrix$ on that chamber. This latter quasi-period in turn divides the least common multiple of the nonzero minors of $\posrootsmatrix$. Finally, one checks that the nonzero minors of $\posrootsmatrix$ are $\pm 1, \pm 2, \pm 3$. Therefore, the minimum quasi-period of $t \mapsto c_{t\lambda, t\nu}^{t\mu}$ is in $\setof{1,2,3,6}$.
\end{proof}

All the possible minimum quasi-periods in \cref{thm: quasi-period at most 6} arise among the stretched Clebsch--Gordan coefficients. An example with minimum quasi-period $6$ appeared in \cref{exm: counterexample to ehrhart positivity} above. The following are examples of the remaining possibilities.

\begin{example}
    We write all dominant weights in the basis of fundamental weights. For each triple of dominant weights, $\PP$ is the set of points $(t_{1}, \dotsc, t_{6}) \in \R^{6}$ satisfying the system of linear equations and inequalities in \cref{thm: explicit equations and inequalities for type G2}, so that $\ehr{\PP}(t) = c_{t\lambda, t\nu}^{t\mu}$ (for $t \in \Zp$) is the corresponding stretched Clebsch--Gordan coefficient by \cref{thm: bz-polytopes}.
    \begin{enumerate}
        \item 
        Let $\lambda \deftobe \nu \deftobe (3, 2)$ and $\mu \deftobe (0, 2)$. Then
        \begin{align*}
            \PP
                =
                \conv
                \big\lbrace
                    & (3, 5, 12, 5, 3, 0),\, (3, 3, 8, 5, 7, 2),\, (3, \tfrac{7}{2}, 9, \tfrac{11}{2}, 6, 1), \\
                    & (3, \tfrac{11}{3}, 8, 5, 7, \tfrac{4}{3}),\, (3, 4, 9, 4, 6, 2) 
                \big\rbrace
        \end{align*}
        and
        \begin{equation*}
            \ehr{\PP}(t)
                =
                \tfrac{2}{3} t^3 + \tfrac{3}{2}t^2 + \tfrac{11}{6}t + 1.
        \end{equation*}

        \item
        Let $\lambda \deftobe \mu \deftobe \nu \deftobe (0,1)$. Then
        \begin{equation*}
            \PP 
                = 
                \conv
                \setof{
                    (0, 1, 3, 1, 0, 0),\, 
                    (0, \tfrac{1}{2}, \tfrac{3}{2}, 1, \tfrac{3}{2}, \tfrac{1}{2})
                }
        \end{equation*}
        and
        \begin{equation*}
            \ehr{\PP}(t)
            =
                \begin{cases*}
                    \tfrac{1}{2}t + 1            & if $t \cong 0 \mod 2$, \\[1ex]
                    \tfrac{1}{2}t + \tfrac{1}{2} & if $t \cong 1 \mod 2$.
                \end{cases*}
        \end{equation*}
        
        \item 
        Let $\lambda \deftobe (2,0)$ and $\mu \deftobe \nu \deftobe (0,1)$. Then
        \begin{equation*}
            \PP 
                = 
                \conv
                \setof{
                    (2, 1, 2, 1, 0, 0),\,
                    (2, \tfrac{4}{3}, 2, \tfrac{2}{3}, 0, 0)
                }
        \end{equation*}
        and
        \begin{equation*}
            \ehr{\PP}(t)
                =
                \begin{cases*}
                    \tfrac{1}{3}t + 1 & if $t \cong 0 \mod 3$, \\[1ex]
                    \tfrac{1}{3}t + \tfrac{2}{3} & if $t \cong 1 \mod 3$, \\[1ex]
                    \tfrac{1}{3}t + \tfrac{1}{3} & if $t \cong 2 \mod 3$.
                \end{cases*}
        \end{equation*}
    \end{enumerate}
\end{example}

\section{Denominators of vertices of BZ polytopes of type \texorpdfstring{$\G$}{G\unisubtwo}}
\label{sec: denominators of vertices}

In this section, we describe an algorithm for computing the denominators that the vertices of Berenstein--Zelevinsky polytopes $\bzlmnc$ may have. Recall that, for a rational polytope $\PP \subset \R^{n}$, the denominator of a vertex $v \in \vertx{\PP}$ is $\den(v) \deftobe \min\setof{d \in \Zp \st dv \in \Z^{n}}$, and the denominator of $\PP$ itself is $\den(\PP) \deftobe \min\setof{d \in \Zp \st \vertx(d\PP) \subset \Z^{n}} = \lcm \setof{\den(v) \st v \in \vertx{\PP}}$.

An implementation of our algorithm in \sagemath{} (see Appendix~\ref{appendix: denominators of vertices code}) establishes \cref{thm: possible denominators - intro version}, which provides the denominators of all vertices of all BZ polytopes $\bzlmnc$ as defined in \cref{defn: c-vec polytope} with $\i = (1,2,1,2,1,2)$. We repeat \cref{thm: possible denominators - intro version} for the convenience of the reader:

\begin{theorem}\label{thm: possible denominators}
    Let $\g$ be a type-$\G$ Lie algebra, and let $\lambda$, $\mu$, and $\nu$ be dominant weights of~$\g$. Then $\den(v) \in \setof{1,2,3,4,5,6,9}$ for every vertex $v$ of $\bzlmnc$, and hence the denominator of $\bzlmnc$ divides $180$.
\end{theorem}

The following example shows that in fact all possible denominators of vertices of such BZ polytopes appear as denominators of the \emph{coordinates} of the vertices of a \emph{single} such BZ polytope. In particular, the upper bound of $180$ on the denominators of such BZ polytopes is tight.

\begin{example}\label{exm: denominator 180}
    It follows from \cref{thm: possible denominators} that $180$ is an upper bound on the maximum possible denominator of the BZ polytope defined in \cref{defn: c-vec polytope} with $\i = (1,2,1,2,1,2)$. The following example shows that this upper bound is sharp. Let $\lambda \deftobe (2, 1)$, $\mu \deftobe (2, 1)$, and $\nu \deftobe (1, 1)$ in the basis of fundamental weights. A computation in \sagemath{} reveals that the vertices of $\bzlmnc$ in this case are as follows:
    \begin{equation*}
        \begin{array}{@{}l@{\quad}l@{\quad}l@{\quad}l@{}}
            (0,1,3,\tfrac53,2,\tfrac13),
            & (0,1,3,\tfrac43,2,\tfrac23),
            & (0,\tfrac56,\tfrac52,\tfrac53,\tfrac52,\tfrac12),
            & (1,2,4,1,0,0),
            \\[0.8em]
            (1,1,3,2,1,0),
            & (\tfrac45,\tfrac45,\tfrac{12}5,\tfrac85,\tfrac95,\tfrac35),
            & (\tfrac53,\tfrac53,\tfrac53,\tfrac{10}9,\tfrac53,\tfrac29),
            & (2,2,3,1,0,0),
            \\[0.8em]
            (2,2,2,1,1,0),
            & (2,2,2,\tfrac23,1,\tfrac13),
            & (2,\tfrac73,3,\tfrac23,0,0),
            & (\tfrac12,\tfrac34,\tfrac94,\tfrac32,\tfrac94,\tfrac34),
            \\[0.8em]
            (\tfrac12,1,3,\tfrac32,\tfrac32,\tfrac12),
            & (1,1,2,\tfrac43,2,\tfrac23),
            & (\tfrac13,\tfrac43,\tfrac{11}3,\tfrac43,1,\tfrac13),
            & (\tfrac12,\tfrac32,\tfrac72,\tfrac32,1,0),
            \\[0.8em]
            (\tfrac12,\tfrac32,4,\tfrac32,\tfrac12,0),
            & (\tfrac13,\tfrac43,\tfrac{11}3,\tfrac53,1,0)
        \end{array}
    \end{equation*}
    Indeed, every possible denominator in \cref{thm: possible denominators} appears among these vertices. Hence, $\bzlmnc$ realizes the maximum possible denominator of such BZ polytopes, namely $180$.
\end{example}

We now describe the algorithm used to prove \cref{thm: possible denominators}. The Ber\-en\-stein--Zelevinsky polytopes $\bzlmnc$ constitute a family of polytopes para\-metrized by nonnegative integral vectors
\begin{equation*}
    {(\lambda_{1}, \lambda_{2}, \mu_{1}, \mu_{2}, \nu_{1}, \nu_{2}) \in \Znn^{6}}.
\end{equation*}
(Recall that, in \cref{thm: explicit equations and inequalities for type G2}, the dominant weights $\lambda, \mu, \nu$ are given in the basis of fundamental weights.)
To describe our algorithm in the general case, fix integral matrices $A, B, C, D$ with $A$~being~$s \times n$, $B$~being~$s \times k$, $C$~being~$r \times n$, and $D$~being~$r \times k$. Assume that $C$ has full row rank (as the corresponding matrix does in the case of the BZ polytopes). Let
\begin{equation*}
    \PP_{t}
        \deftobe
        \setof{
            x \in \R^{n}
            \st
            Ax \le Bt,\,
            Cx = Dt
        } 
        \qquad \text{for all $t \in \Z^{k}$ with $t \ge 0$},
\end{equation*}
and let $\FF \deftobe \setof{\PP_{t} \st t \in \Znn^{k}}$ and $\vertx(\FF) \deftobe \bigcup_{\PP \in \FF} \vertx(\PP)$. We seek the set 
\begin{equation}\label{eq: set of denominators of a family of polyhedra}
    \DD
        \deftobe
        \setof{
            \den(v) 
            \st 
            \text{%
                $v \in \vertx(\FF)$%
            }
        }.
\end{equation}

Our approach builds on
\cite[Section~4.1]{VSBLB2007}
and
\cite{BBDKV2019}, especially, in the latter, their Definition~2.9 and Equation~(13), and the discussion following their Definition~2.25.
The treatment in \cite{BBDKV2019} applies to less-constrained families of the form $\setof{\PP_{(b,d)}}$ with 
\begin{equation*}
    \PP_{(b, d)}
        \deftobe
        \setof{
            x \in \R^{n}
            \st
            Ax \le b,\,
            Cx = d
        } 
        \; \text{for arbitrary real right-hand sides $b, d$}.
\end{equation*}
We adapt their methods to find the denominators of the vertices of just those polyhedra of the form $\PP_{t} = \PP_{(Bt, Dt)}$ with $t \in \Z^{k}$ nonnegative. 

For each subset $I \subset [s]$, let $A_{I}$, respectively $B_{I}$, be the submatrix of $A$, respectively $B$, consisting of the rows indexed by the elements of $I$, and let 
\begin{equation*}
    M_{I}
        \deftobe
        \begin{bmatrix*}[c]
            C \\
            A_{I}  
        \end{bmatrix*}.
\end{equation*}
Let $\II \deftobe \setof{I \subset [s] \st \text{$M_{I}$ is invertible}}$. For all $I \in \II$, let
\begin{equation*}
    X_{I}
        \deftobe
        M_{I}^{-1}
        \begin{bmatrix*}[c]
            D \\
            B_{I} 
        \end{bmatrix*}.
\end{equation*}
A point $x \in \R^{n}$ is in $\vertx(\FF)$ if and only if $x \in \PP_{t}$ and $x = X_{I}t$ for some $t \in \Znn^{k}$ and $I \in \II$. The condition that $X_{I}t \in \PP_{t}$ is equivalent to $(B - AX_{I})t \ge 0$. Conversely, if $I \in \II$, then $X_{I}t$ is a vertex of~$\PP_{t}$ for all $t \in K_{I} \cap \Z^{k}$, where $K_{I}$ is the rational polyhedral cone
\begin{equation*}
    K_{I}
        \deftobe
        \setof{t \in \R^{k} \st \text{$t \ge 0 $ and $(B - AX_{I})t \ge 0$}}.
\end{equation*}
In other words, putting $S_{I} \deftobe K_{I} \cap \Z^{k}$,
\begin{equation*}
    \vertx(\FF) 
        = 
        \bigcup_{I \in \II}
        \setof{X_{I}t \st t \in S_{I}},
\end{equation*}
and so 
\begin{equation*}
    \DD
        =
        \bigcup_{I \in \II}
        \setof{\den(X_{I} t) \st t \in S_{I}}.
\end{equation*}

Now, for all $I \in \II$, let $L_{I} \deftobe \spn_{\R}(K_{I}) \cap \Z^{k}$, and consider the denominator of a point of the form $X_{I} t \in \Q^{n}$ with $t \in L_{I}$. This denominator is the order of the element represented by~$X_{I}t$ in the finite abelian group $G_{I} \deftobe X_{I}L_{I}/(X_{I}L_{I} \cap \Z^{n})$. The set of orders of elements of a finite abelian group is the set of divisors of that group's exponent. Thus, $\setof{\den(X_{I}t) \st t \in L_{I}}$ is the set of divisors of the exponent of $G_{I}$. This exponent, in turn, is the denominator~$d_{I}$ of the matrix $X_{I}Z_{I}$ (\emph{i.e.}, the least positive integer $d_{I}$ such that $d_{I}X_{I}Z_{I}$ is integral), where $Z_{I}$ is a matrix whose columns are a $\Z$-basis of $L_{I}$ (setting $d_{I} \deftobe 1$ if $L_{I} = \setof{0}$).  Thus,
\begin{equation*}
    \setof{
        \den(X_{I}t)
        \st
        t \in L_{I}
    }
        = 
        \Div(d_{I}),
\end{equation*}
where $\Div(d_{I})$ denotes the set of divisors of $d_{I}$.

Finally, we claim that $\setof{\den(X_{I}t) \st t \in L_{I}} = \setof{\den(X_{I}t) \st t \in S_{I}}$. This is the case because the cone $K_{I}$ has nonempty interior relative to its span, and $X_{I}$ is rational, so the semigroup~$S_{I}$ contains an element of every fiber of the composition $L_{I} \to X_{I} L_{I} \to X_{I} L_{I}/(X_{I} L_{I} \cap \Z^{n})$.
Therefore,
\begin{equation*}
    \DD
        =
        \bigcup_{I \in \II}
        \Div(d_{I}).
\end{equation*}

The preceding remarks establish the following algorithm for computing the set $\DD$. (We continue to use the notation fixed above.)
\begin{enumerate}
    \item 
    For each $I \in \II$,
    \begin{enumerate}
        \item
        compute a rational matrix $N_{I}$ such that $\ker(N_{I}) = \spn_{\R}(K_{I})$;
        
        \item 
        compute a $\Z$-basis of the lattice $L_{I} \deftobe \ker(N_{I}) \cap \Z^{k}$, and write this $\Z$-basis as the columns of a matrix $Z_{I}$;
        
        \item 
        let $d_{I} \deftobe \den(X_{I} Z_{I})$ (with $d_{I} \deftobe 1$ if $L_{I} = \setof{0}$).
    \end{enumerate}
    
    \item 
    Return
    \begin{equation*}
        \DD
            =
            \bigcup_{I \in \II}
            \Div(d_{I}).
    \end{equation*}
\end{enumerate}

\appendix

\section{SageMath code for finding all denominators of all BZ polytopes}
\label{appendix: denominators of vertices code}

The code below is an implementation in \sagemath{} of the algorithm described in \cref{sec: denominators of vertices}. In particular, the code defines a function that returns the set of all denominators of vertices of the polyhedra in a parametrized family $\setof{\PP_{t} \st t \in \Znn^{k}}$, where
\begin{equation*}
    \PP_{t}
        \deftobe
        \setof{
            x \in \R^{n}
            \st
            Ax \le Bt,\,
            Cx = Dt
        } 
        \qquad \text{for all $t \in \Z^{k}$ with $t \ge 0$}.
\end{equation*}
Also included is code that applies this function to the family of BZ polytopes $\bzlmnc$ defined in \cref{defn: c-vec polytope} with $\i = (1,2,1,2,1,2)$. The output of this code establishes \cref{thm: possible denominators}.

\begin{lstlisting}[
    style=sagecode,
    caption={SageMath code for computing possible vertex denominators},
    label={lst: G2 vertex denominators}
]
"""Exact vertex-denominator spectrum for the type-G2 BZ polytope family.

Run with SageMath (for example: sage -python G2_vertex_denominators_appendix_builtin.py).
All computations are exact; no floating-point arithmetic is used.
"""

from sage.all import *
from itertools import combinations


def possible_denominators(A, B, C, D):
    """Return the possible denominators of vertices of

        P_t = {x : A*x <= B*t, C*x = D*t},

    where t ranges over nonnegative integral vectors.
    """
    s, n = A.dimensions()
    r = C.nrows()
    k = B.ncols()

    assert C.dimensions() == (r, n) and C.rank() == r
    assert B.dimensions() == (s, k)
    assert D.dimensions() == (r, k)

    # Construct the indexing set II.
    II = []

    for I in combinations(range(s), n - r):
        M_I = C.stack(A.matrix_from_rows(I))
        if M_I.det() != 0:
            II.append(I)

    # Construct the set DD of all denominators
    DD = set()

    # For every I in II, add the denominators contributed by I to DD.
    for I in II:
        M_I = C.stack(A.matrix_from_rows(I))
        R_I = D.stack(B.matrix_from_rows(I))
        X_I = M_I.change_ring(QQ).solve_right(R_I)

        # An integral X_I contributes only denominator 1.
        if X_I.denominator() == 1:
            d_I = ZZ(1)
        else:
            # Construct K_I = {t >= 0 : (B - A*X_I)*t >= 0}.
            inequalities = identity_matrix(QQ, k).stack(B - A*X_I)
            K_I = Polyhedron(
                ieqs=[[0] + list(row) for row in inequalities.rows()],
                ambient_dim=k,
                base_ring=QQ,
            )

            # (a) Find N_I such that ker(N_I) = span_R(K_I).
            KI_equations = K_I.equations()

            if len(KI_equations) == 0:
                # K_I spans R^k, so L_I = Z^k.
                Z_I = identity_matrix(ZZ, k)

                # (c) Compute d_I = den(X_I * Z_I).
                d_I = (X_I * Z_I).denominator()

            else:
                # A() returns the coefficients of the variables in a linear equation.
                N_I = matrix(QQ, [eq.A() for eq in KI_equations])

                # (b) Compute the lattice L_I = ker(N_I) intersect Z^k.
                L_I = N_I.integer_kernel()

                if L_I.rank() == 0:
                    # If L_I = {0}, then set d_I = 1.
                    d_I = ZZ(1)
                else:
                    # Write a Z-basis of L_I as the columns of Z_I.
                    Z_I = L_I.basis_matrix().transpose()

                    # (c) Compute d_I = den(X_I * Z_I).
                    d_I = (X_I * Z_I).denominator()

        DD = DD.union(d_I.divisors())

    return sorted(DD)


# Type-G2 input matrices. The rows displayed here are the rows of -A.
A = -matrix(ZZ, [
    [ 1,  0,  0,  0,  0,  0],
    [ 0,  3, -1,  0,  0,  0],
    [ 0,  0,  2, -3,  0,  0],
    [ 0,  0,  0,  3, -2,  0],
    [ 0,  0,  0,  0,  1, -3],
    [ 0,  0,  0,  0,  0,  1],
    [-1,  0,  0,  0,  0,  0],
    [ 1, -1,  0,  0,  0,  0],
    [ 0,  2, -1,  0,  0,  0],
    [ 0,  0,  1, -2,  0,  0],
    [ 0,  0,  0,  1, -1,  0],
    [ 0,  0,  0,  0,  0, -1],
    [-1,  3, -2,  3, -2,  3],
    [ 0, -1,  1, -2,  1, -2],
    [ 0,  0, -1,  3, -2,  3],
    [ 0,  0,  0, -1,  1, -2],
    [ 0,  0,  0,  0, -1,  3],
    [ 0,  0,  0,  0,  0, -1]
])

C = matrix(ZZ, [
    [1, 0, 1, 0, 1, 0],
    [0, 1, 0, 1, 0, 1]
])

# t = (lambda_1, lambda_2, mu_1, mu_2, nu_1, nu_2).
D = matrix(ZZ, [
    [2, 3, -2, -3, 2, 3],
    [1, 2, -1, -2, 1, 2]
])

B = matrix(ZZ, 18, 6)  # Initially all entries are zero.
B[6, 0] = 1             # Seventh RHS is lambda_1.
for i in range(7, 12):
    B[i, 1] = 1         # Eighth through twelfth RHSs are lambda_2.
for i in (12, 14, 16):
    B[i, 4] = 1         # Alternating RHSs are nu_1.
for i in (13, 15, 17):
    B[i, 5] = 1         # Remaining RHSs are nu_2.

DD = possible_denominators(A, B, C, D)
print("Possible vertex denominators:", DD)
print("LCM of possible vertex denominators:", lcm(DD))
\end{lstlisting}

\printbibliography

\end{document}